\documentclass[12pt]{article}

\usepackage{graphicx,psfrag}
\usepackage{amssymb,amsfonts,amsthm}
\usepackage{amsmath,amsthm,amssymb}
\usepackage{amsfonts}
\usepackage{relsize}
\usepackage[T1]{fontenc}
\usepackage[utf8]{inputenc}
\usepackage{graphicx}
\usepackage{color}
\usepackage{subcaption,mathtools}

\newcommand{\Stab}{\mathrm{Stab}}
\newcommand{\Supp}{\mathrm{Supp}}

\newcommand{\la}{\langle}
\newcommand{\ra}{\rangle}

\newcommand{\PL}{\widetilde{\mathrm{PL}}}
\newcommand{\pl}{\mathrm{PLF}}

\newcommand{\Homeo}{\mathrm{Homeo}}
\newcommand{\zz}{\Z[\frac 12]}

\newcommand{\F}{\mathcal{F}}

\newcommand{\dist}{{\mathrm{dist}}}

\newcommand{\be}{\begin{equation}}
\newcommand{\ee}{\end{equation}}

\newcommand {\N}{\mathbb{N}} %% positive integers
\newcommand {\Z}{\mathbb{Z}}            %% integers
\newcommand {\R}{\mathbb{R}} %% reals
\newcommand {\iv}{^{-1}}

\numberwithin{equation}{section}

\newcounter{AbcT}

\newcommand{\nc}{\newcommand}
\nc{\meet}{\wedge}
\nc{\op}{\operatorname}\nc{\FP}{\op{FP}}\nc{\FS}{\op{FS}}\nc{\FPhat}{\widehat{\op{FP}}}

\newtheorem {Theorem}    {Theorem}[section]

\newtheorem {Problem}    [Theorem]{Problem}

\newtheorem {Lemma}      [Theorem]    {Lemma}
\newtheorem {Corollary}   [Theorem] {Corollary}

\theoremstyle{remark}
\newtheorem {Remark}		 [Theorem]    {\bf{Remark}}

\newtheorem {Definition} [Theorem]    {\bf{Definition}}

\newcommand{\1}{\mathbf{1}}

\begin{document}

\title{On the stabilizers of finite sets of numbers in the R. Thompson group $F$}

\author{Gili Golan, Mark Sapir\thanks{The research of the first author was supported in part by a Fulbright grant and a post-doctoral scholarship of Bar-Ilan University, the research of the second author was supported in part by the NSF grants DMS 1418506,
DMS	1318716.}}
\date{}
\maketitle

\hskip .7 in Dedicated to Professor Yuri Burago on the occasion
     of his $80$th birthday

\abstract{We study subgroups $H_U$ of the R. Thompson group $F$ which are stabilizers of finite sets $U$ of numbers in the interval $(0,1)$. We describe the algebraic structure of $H_U$ and prove that the stabilizer $H_U$ is finitely generated if and only if $U$ consists of rational numbers. We also show that  such subgroups are isomorphic surprisingly often. In particular, we prove that if finite sets $U\subset [0,1]$ and $V\subset [0,1]$ consist of rational numbers which are not finite binary fractions, and $|U|=|V|$, then the stabilizers of $U$ and $V$ are isomorphic. In fact these subgroups are conjugate inside a subgroup $\bar F<\Homeo([0,1])$ which is the completion of $F$ with respect to what we call the Hamming metric on $F$. Moreover the conjugator can be found in a certain subgroup $\F < \bar F$ which consists of possibly infinite tree-diagrams with finitely many infinite branches. We also show that the group $\F$ is non-amenable.}

\tableofcontents

\section{Introduction}

R. Thompson group $F$ is one of the most interesting infinite finitely generated groups. It is usually defined as a group of piecewise linear increasing homeomorphisms of the unit interval with all break points of derivative finite dyadic fractions (i.e., numbers from $\Z[\frac 12]$) and all slopes powers of $2$. The group has many other descriptions (for some of them see Section \ref{sec:pre}). The group $F$ is finitely presented, does not have free non-cyclic subgroups, and satisfies many other remarkable properties which are the subject of numerous papers. One of the main questions about $F$ is whether it is amenable (the problem was first mentioned in print by Ross Geoghegan). Incorrect proofs of amenability and non-amenability of $F$ are published quite often, and some of these papers (despite having incorrect proofs) are quite interesting because they show deep connections of $F$ with diverse branches of mathematics. For example, \cite{Shav} shows why a mathematical physicist would be interested in R. Thompson group $F$, paper \cite{Moore} shows a deep connection with logic and Ramsey theory, and \cite{Wajn} shows a connection with graphs on surfaces. Quite recently Vaughan Jones discovered a striking connection between $F$ and planar algebras, subfactors and the knot theory \cite{Jones}. It turned out that just like braid groups, elements of $F$ can be used to construct all links. He also considered several linearized permutational representations of $F$ on the Schreier graphs of some subgroups of $F$ including the subgroup  $\overrightarrow F$ defined in terms of the corresponding sets of links. This motivated our renewed interest in subgroups of $F$ \cite{GS, GolanSapir}. It turned out that Jones' subgroup $\overrightarrow F$ is quite interesting. For example,  even though it is not maximal, there are finitely many (exactly two) subgroups of $F$ bigger than $\overrightarrow F$, i.e. $\overrightarrow F$ is of \emph{quasi-finite index} in $F$ \cite{GolanSapir}.

Other subgroups of interest include the stabilizers $H_U$ of finite sets $U\subset (0,1)$ first considered by Savchuk \cite{Sav1, Sav} who proved that the Schreier graph of $F/H_U$ is amenable for every finite $U$. In \cite{GolanSapir} we showed that each $H_U$ is also of quasi-finite index: every subgroup of $F$ containing $H_U$ is of the form $H_V$ where $V\subseteq U$.

Savchuk \cite{Sav1}  noticed that if $U$ contains an irrational number, then $H_U$ is not finitely generated. It is easy to see that if $U$ consists of numbers from $\zz$, then $H_U$ is isomorphic to the direct product of $|U|+1$ copies of $F$ and hence is $(2|U|+2)$-generated. Thus if $U\subset \zz$, then the isomorphism class of $H_U$ depends only on the size of $U$. In fact if $|U|=|V|$ and $U,V \subset \zz$, then it is easy to see that $H_U$ and $H_V$ are conjugate in $F$. This leaves open the following.

\begin{Problem}\label{p:7}
What is the structure of $H_U$ for an arbitrary finite $U$? When is $H_U$ finitely generated? When are $H_U$ and $H_V$ isomorphic?
\end{Problem}

In this paper, we continue the study of subgroups $H_U$ and answer the questions from Problem \ref{p:7}.  Every subset $U$ of $(0,1)$ is naturally subdivided into three subsets $U=U_1\cup U_2\cup U_3$ where $U_1$ consists of numbers from $\Z[\frac12]$ (i.e., numbers of the form $.u$ where $u$ is a finite word in $\{0,1\}$), $U_2$ consists of rational numbers not in $\zz$  (i.e., numbers  of the form $.ps^\N$ where $p, s$ are finite binary words and $s$ contains both digits $0$ and $1$), and $U_3$ consists of irrational numbers. We shall call $U_1, U_2, U_3$ the \emph{natural partition} of $U$. We show that $H_U$ is finitely generated  if and only if $U$ consists of rational numbers, that is, $U_3$ is empty. In that case we find the minimal number of generators of $H_U$ and classify subgroups $H_U$ up to isomorphism. In particular, we show (Theorem \ref{th:m1}) that if $U_1=U_3=\emptyset$, then, up to isomorphism, $H_U$ depends only on the size $|U|$. For example, $H_{\{\frac13\}}$ is isomorphic to $H_{\{\frac17\}}$. Moreover if $\tau(U)$ is the \emph{type}  of $U$ i.e. the word in the alphabet $\{1,2,3\}$ where the $i$-letter of $\tau(U)$ is $1$ if the $i$-th number in $U$ (with respect to natural order) is from $\zz$, $2$ if the $i$-th number in $U$ is rational not from $\zz$, and $3$ if the $i$-th number is irrational, then $H_U$ is isomorphic to $H_V$ provided $\tau(U)\equiv \tau(V)$  \footnote{$p\equiv q$ denotes letter-by-letter equality of words $p, q$.} (Theorem \ref{th:m1}).

Note that $H_{\{\frac 13\}}$ and $H_{\{\frac 17\}}$ are obviously not conjugate in $F$ or even in $\PL_2(\R)$, the group of piecewise linear homeomorphisms of $\R$ with all break points in $\zz$ and all absolute values of slopes powers of 2 (and so not in $\mathrm{Aut}(F)$ \cite{BrinCh}). We are going to prove (see Theorem \ref{t:pl}) that if $\tau(U)\equiv \tau(V)$, then $H_U$ and $H_V$ are conjugate in $\Homeo([0,1])$ and, in fact, in a relatively small subgroup $\F$ of $\Homeo([0,1])$. We will also show that one can construct a completion $\bar F$ of $F$ with respect to a certain metric (which is similar to the Hamming metric on the symmetric group $S_n$) and show that the natural embedding $F\to \bar F$ extends to an embedding $\F\to \bar F$. The groups $\F$ and $\bar F$ are interesting on their own. We prove, in particular, that $\F$ contains a non-abelian free subgroup, so $\bar F$ is a non-amenable completion of $F$.

\begin{Remark} Note  that Theorem \ref{th:m1} (the isomorphism theorem for some subgroups $H_U$) follows from Theorem \ref{t:pl} (the conjugacy theorem for some subgroups $H_U$). Nevertheless we decided to keep Theorem \ref{th:m1} because its proof gives additional algebraic information about subgroups $H_U$. %In particular, we prove that $H_U$ is a multiple ascending HNN extension of a finite direct power of the derived subgroup of $F$ (see Lemma \ref{HNN-ext}).
%******* In the lemma it is an ascending HNN-extension of a direct power of F, but that is not always the case.
\end{Remark}

Here is a combination of several results proved in this paper:

\begin{Theorem}\label{t:01} Let $U$ be a finite set of numbers from $(0,1)$ and let $U=U_1\cup U_2\cup U_3$ be the standard partition, $r=|U|$, $m_i=|U_i|$, $i=1,2,3$. Then
\begin{enumerate}

\item (Theorem \ref{thm:semi})
 $H_U$ is isomorphic to a semidirect product
$$H_U\cong [F,F]^{r+1}\rtimes \mathbb{Z}^{2m_1+m_2+2}.$$

\item (Theorem \ref{t:rat}) The subgroup $H_U$ is finitely generated if and only if $U_3$ is empty (that is, $U$ consists of rational numbers). In that case the smallest number of generators of $H_U$ is $2m_1+m_2+2$.
\item (Theorem \ref{t:dist}) Let $U_3=\emptyset$. Then the subgroup $H_U$ is undistorted in $F$.

\end{enumerate}

\end{Theorem}

In Section \ref{open}, we list some open problems.

\begin{Remark} After the first versions of our paper appeared on arXiv, Ralph Strebel informed us that several results of this paper, in particular Theorem 3.2, can be proved for generalizations of the R. Thompson groups considered in \cite{Bieri}.
\end{Remark}

\section{Preliminaries on $F$}\label{sec:pre}

\subsection{$F$ as a group of homeomorphisms}\label{ss:nf}

Recall that $F$ consists of all piecewise-linear increasing self-homeomorphisms of the unit interval with slopes of all linear pieces powers of $2$ and all break points of the derivative in $\zz$. The group $F$ is generated by two functions $x_0$ and $x_1$ defined as follows \cite{CFP}.

	\[
   x_0(t) =
  \begin{cases}
   2t &  \hbox{ if }  0\le t\le \frac{1}{4} \\
   t+\frac14       & \hbox{ if } \frac14\le t\le \frac12 \\
   \frac{t}{2}+\frac12       & \hbox{ if } \frac12\le t\le 1
  \end{cases} 	\qquad	
   x_1(t) =
  \begin{cases}
   t &  \hbox{ if } 0\le t\le \frac12 \\
   2t-\frac12       & \hbox{ if } \frac12\le t\le \frac{5}{8} \\
   t+\frac18       & \hbox{ if } \frac{5}{8}\le t\le \frac34 \\
   \frac{t}{2}+\frac12       & \hbox{ if } \frac34\le t\le 1 	
  \end{cases}
\]

The composition in $F$ is from left to right.

Every element of $F$ is completely determined by how it acts on the set $\zz$. Every number in $(0,1)$ can be described as $.s$ where $s$ is an infinite word in $\{0,1\}$. For each element $g\in F$ there exists a finite collection of pairs of (finite) words $(u_i,v_i)$ in the alphabet $\{0,1\}$ such that every infinite word in $\{0,1\}$ starts with exactly one of the $u_i$'s. The action of $F$ on a number $.s$ is the following: if $s$ starts with $u_i$, we replace $u_i$ by $v_i$. For example, $x_0$ and $x_1$  are the following functions:

\[
   x_0(t) =
  \begin{cases}
   .0\alpha &  \hbox{ if }  t=.00\alpha \\
    .10\alpha       & \hbox{ if } t=.01\alpha\\
   .11\alpha       & \hbox{ if } t=.1\alpha\
  \end{cases} 	\qquad	
   x_1(t) =
  \begin{cases}
   .0\alpha &  \hbox{ if } t=.0\alpha\\
   .10\alpha  &   \hbox{ if } t=.100\alpha\\
   .110\alpha            &  \hbox{ if } t=.101\alpha\\
   .111\alpha                      & \hbox{ if } t=.11\alpha\
  \end{cases}
\]
where $\alpha$ is any infinite binary word.
For the generators $x_0,x_1$ defined above, the group $F$ has the following finite presentation \cite{CFP}.
$$F=\la x_0,x_1\mid [x_0x_1^{-1},x_1^{x_0}]=1,[x_0x_1^{-1},x_1^{x_0^2}]=1\ra,$$ where $a^b$ denotes $b\iv ab$.

Sometimes, it is more convenient to consider an infinite presentation of $F$. For $i\ge 1$, let $x_{i+1}=x_0^{-i}x_1x_0^i$. In these generators, the group $F$ has the following presentation \cite{CFP}
$$\la x_i, i\ge 0\mid x_i^{x_j}=x_{i+1} \hbox{ for every}\ j<i\ra.$$

\subsection{Elements of F as pairs of binary trees} \label{sec:red}

Often, it is more convenient to describe elements of $F$ using pairs of finite binary trees drawn on a plane. Trees are considered up to isotopies of the plane. Elements of $F$ are  pairs of full finite binary trees $(T_+,T_-)$ which have the same number of leaves. Such a pair will sometimes be called a \emph{tree-diagram}.

If $T$ is a (finite or infinite) binary tree, a \emph{branch} in $T$ is a maximal simple path starting from the root. Every non-leaf vertex of $T$ has two outgoing edges: the left edge and the right edge.
If every left edge of $T$ is labeled by $0$ and every right edge is labeled by $1$, then every branch of $T$ is labeled by a (finite or infinite) binary word $u$.  We will usually ignore the distinction between a branch and its label.

Let $(T_+,T_-)$ be a tree-diagram where $T_+$ and $T_-$ have $n$ leaves. Let $u_1,\dots,u_n$ (resp. $v_1,\dots,v_n$) be the branches of $T_+$ (resp. $T_-$), ordered from left to right.
%The finite collection of pairs $(u_i,v_i)$ $i=1,\dots,n$ gives the description of the element as a function in binary form as above.
For each $i=1,\dots,n$ we say that the tree-diagram $(T_+,T_-)$ has a \emph{pair of branches} $u_i\rightarrow v_i$. The function $g$ from $F$ corresponding to this tree-diagram takes binary fraction $.u_i\alpha$ to $.v_i\alpha$ for every $i$ and every infinite binary word $\alpha$. We will also say that the element $g$ takes the branch $u_i$ to the branch $v_i$.
The tree-diagrams of the generators of $F$, $x_0$ and $x_1$, appear in Figure \ref{fig:x0x1}.

\begin{figure}[ht]
\centering
\begin{subfigure}{.5\textwidth}
  \centering
  \includegraphics[width=.5\linewidth]{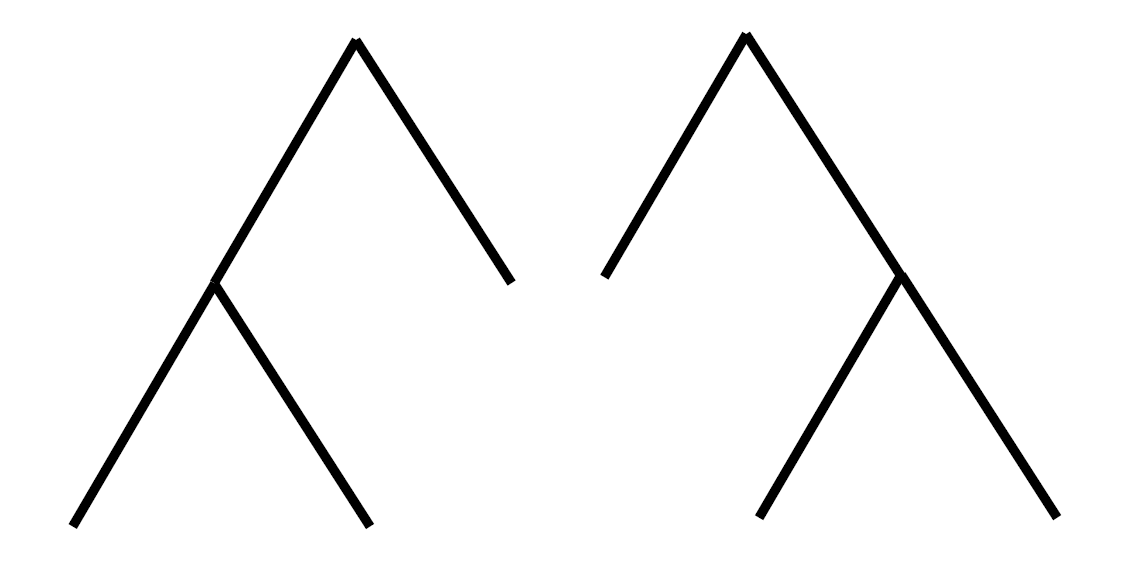}
  \caption{}%The tree diagram $(T_+,T_-)$ of $x_0$}
  \label{fig:x0}
\end{subfigure}%
\begin{subfigure}{.5\textwidth}
  \centering
  \includegraphics[width=.5\linewidth]{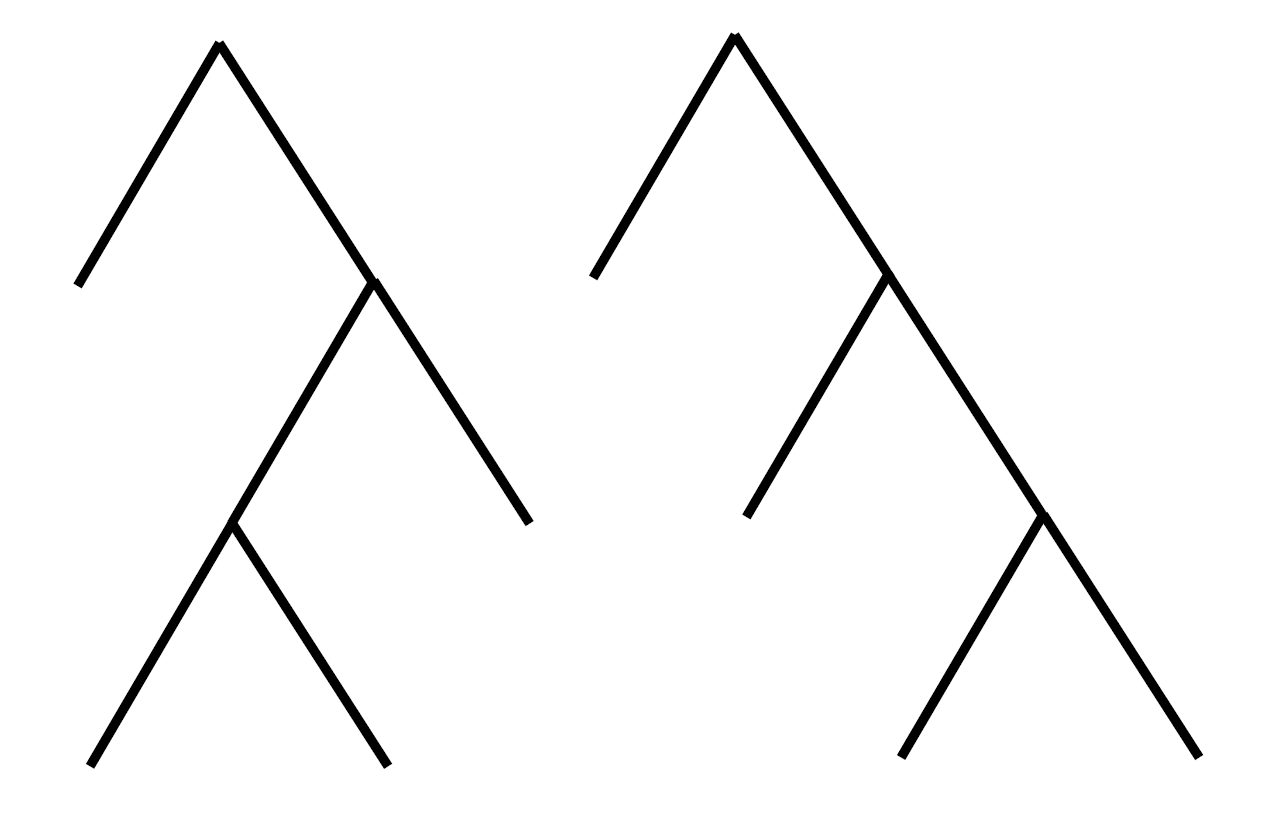}
  \caption{}%The tree diagram $(T_+,T_-)$ of $x_1$}
  \label{fig:x1}
\end{subfigure}
\caption{(a) The tree-diagram of $x_0$. (b) The tree-diagram of $x_1$. In both figures, $T_+$ is on the left and $T_-$ is on the right.}
\label{fig:x0x1}
\end{figure}

A \emph{caret} is a binary tree composed of a single vertex with two children. If $(T_+,T_-)$ is a tree-diagram, then attaching a caret to the $i$-th leaf of both $T_+$ and $T_-$ does not affect the function in $F$ represented by the tree-diagram $(T_+,T_-)$. The inverse action of \emph{reducing} common carets does not affect the function either (the pair $(T_+,T_-)$ has a \emph{common caret} if leaves number $i$ and $i+1$ have a common father in $T_+$ as well as in $T_-$).  Two pairs of trees $(T_+,T_-)$ and $(R_+,R_-)$ are said to be \emph{equivalent} if one results from the other by a finite sequence of inserting and reducing common carets. If $(T_+,T_-)$ does not have a common caret then $(T_+,T_-)$ is said to be \emph{reduced}. Every tree-diagram is equivalent to a unique reduced tree-diagram. Thus elements of $F$ can be represented uniquely by reduced tree-diagrams \cite{CFP}.

An alternative way of describing the function in $F$ corresponding to a given tree-diagram is the following. For each finite binary word $u$, we let the \emph{interval associated with $u$}, denoted by $[u]$, be the interval $[.u,.u1^\N]$. If $(T_+,T_-)$ is a tree-diagram for $f\in F$, we let $u_1,\dots,u_n$ be the branches of $T_+$ and $v_1,\dots,v_n$ be the branches of $T_-$. Then the intervals $[u_1],\dots,[u_n]$ (resp. $[v_1],\dots,[v_n]$) form a subdivision of the interval $[0,1]$. The function $f$ maps each interval $[u_i]$ linearly onto the interval $[v_i]$.

Below, when we say that a function $f$ has a pair of branches $u_i\rightarrow v_i$, the meaning is that some tree-diagram representing $f$ has this pair of branches. In other words, this is equivalent to saying that $f$ maps $[u_i]$ linearly onto $[v_i]$.

\begin{Remark}[See \cite{CFP}]\label{r:000}
The tree-diagram where both trees are just singletons plays the role of identity in $F$. Given a tree-diagram $(T_+^1,T_-^1)$, the inverse tree-diagram is $(T_-^1,T_+^1)$. If $(T_+^2,T_-^2)$ is another  tree-diagram then the product of $(T_+^1,T_-^1)$ and $(T_+^2,T_-^2)$ is defined as follows. There is a minimal finite binary tree $S$ such that $T_-^1$ and $T_+^2$ are rooted subtrees of $S$ (in terms of subdivisions of $[0,1]$, the subdivision corresponding to $S$ is the intersection of the subdivisions corresponding to $T_-^1$ and $T_+^2$). Clearly, $(T_+^1,T_-^1)$ is equivalent to a tree-diagram $(T_+,S)$ for some finite binary tree $T_+$. Similarly, $(T_+^2,T_-^2)$ is equivalent to a tree-diagram $(S,T_-)$. The \emph{product}  $(T_+^1,T_-^1)\cdot(T_+^2,T_-^2)$ is (the reduced tree-diagram equivalent to) $(T_+,T_-)$.
\end{Remark}

Obviously, the mapping of tree-diagrams to functions in $F$ respects operations defined in Remark \ref{r:000}.

\subsection{Choosing elements in $F$}

In most proofs in this paper, we choose elements with a given set of pairs of branches, or elements which map certain intervals or numbers from $[0,1]$ in a predetermined way. In doing so, we usually apply the next lemma. It follows directly from the proof of \cite[Lemma 2.1]{BrinCh}.

\begin{Lemma}\label{choice}
Let $[a_1,b_1],\dots,[a_m,b_m]$, $[c_1,d_1],\dots,[c_m,d_m]$ be closed subintervals of $[0,1]$ (possibly of length 0, i.e., points) with endpoints from $\zz$. Assume that the interiors of the intervals $[a_1,b_1],\dots,[a_m,b_m]$ (resp. $[c_1,d_1],\dots,[c_m,d_m]$) are pairwise disjoint and that the intervals $[a_1,b_1],\dots,[a_m,b_m]$ (resp. $[c_1,d_1],\dots, [c_m,d_m]$), considered as sub-intervals of $[0,1]$, are ordered from left to right. Assume in addition that the following conditions are satisfied.
\begin{enumerate}
\item $[a_i,b_i]$ has empty interior if and only if $[c_i,d_i]$ has empty interior.
\item $0\in [a_1,b_1]$ if and only if $0\in [c_1,d_1]$.
\item $1\in [a_m,b_m]$ if and only if $1\in [c_m,d_m]$.
\item for all $i=1,\dots,m-1$, the intervals $[a_i,b_i]$ and $[a_{i+1},b_{i+1}]$ share a boundary point if and only if the intervals $[c_i,d_i]$ and $[c_{i+1},d_{i+1}]$ share a boundary point.
\end{enumerate}
Then there is an element $f\in F$ which maps each interval $[a_i,b_i]$, $i=1,\dots,m$ onto the interval $[c_i,d_i]$. In addition, if for some $i$, $[a_i,b_i]$ has positive length and $\frac{b_i-a_i}{d_i-c_i}$ is an integer power of $2$ (in particular, if both $[a_i,b_i] $ and $[c_i,d_i]$ are dyadic intervals), then $f$ can be taken to map
$[a_i,b_i]$ linearly onto $[c_i,d_i]$.
\end{Lemma}

\begin{Remark}\label{r:choice}
The proof of Lemma 2.1 in \cite{BrinCh} also implies (in the notations of Lemma \ref{choice}) that if for each $i$ we choose an element $g_i\in F$ which maps $[a_i,b_i]$ onto $[c_i,d_i]$, then there is an element $f\in F$ such that for all $i\in\{1,\dots,m\}$, the restriction of $f$ to $[a_i,b_i]$ coincides with $g_i$.
\end{Remark}

\begin{Remark}
Unless explicitly stated otherwise, all closed intervals considered below have positive lengths.
\end{Remark}

\subsection{On branches and fixed points}

In this section we consider the relation between the set of fixed points of an element $f\in F$ and a tree-diagram $(T_+,T_-)$ representing it.
Let $(T_+,T_-)$ be a tree-diagram of an element $f\in F$. Let $u_i\rightarrow v_i$, $i=1,\dots,n$ be the pairs of branches of $(T_+,T_-)$. Since $f$ is linear on each interval $[u_i]$, if the interval is not fixed (i.e., if the words $u_i$ and $v_i$ are different), then the interval $[u_i]$ contains at most one fixed point. This fixed point can be found as follows.

Let $i\in\{1,\dots,n\}$ and assume that $u_i\neq v_i$, We can assume that $|u_i|\le |v_i|$ by replacing $f$ by $f^{-1}$ if necessary.  If $u_i$ is not a prefix of $v_i$ then the intervals $[u_i]$ and $[v_i]$ are disjoint and there are no fixed points in $[u_i]$. Otherwise, $v_i\equiv u_is_i$ for some nonempty suffix $s_i$. The number $\alpha_i=.u_is_i^{\N}$ is the unique number fixed in $[u_i]$. Note that if $s_i\equiv 0^k$ or $s_i\equiv 1^k$ for some $k\in\mathbb{N}$, then $\alpha_i$ is from $\zz$. Otherwise $\alpha_i$ is a rational (but not in $\zz$) fixed point of $f$.

\begin{Corollary}[Savchuk, \cite{Sav1}]\label{cor:irr}
Let $f\in F$ and assume that $f$ fixes an irrational number $\alpha$. Then $f$ fixes a neighborhood $(\alpha-\epsilon, \alpha+\epsilon)$ of $\alpha$.
\end{Corollary}

The discussion above also implies the following.

\begin{Lemma}\label{rational}
Let $f\in F$ and assume that $f$ fixes a rational number $\alpha\notin \zz$. Let $\alpha=.ps^{\N}$ where
 $s$ is a minimal period of $\alpha$. Then $f$ has a pair of branches of the form $ps^{m_1}\rightarrow ps^{m_2}$ for some $m_1,m_2\ge 0$.
\end{Lemma}

\begin{proof}
Let $(T_+,T_-)$ be a tree-diagram for $f$ and let $u\rightarrow v$ be a pair of branches of $(T_+,T_-)$ such that $\alpha\in [u]$. Since $\alpha$ is fixed by $f$, $\alpha\in [v]$ and we can assume that $|u|\le |v|$. Similarly, by considering a non reduced form of $(T_+,T_-)$ we can assume that $u=ps^{m_1}$ for some $m_1\ge 0$. If $v=u$ we are done. Otherwise, let $v=uw$ for some nonempty word $w$. The unique fixed point of $f$ in $[u]$ is then $.uw^{\N}$. Thus $\alpha=.ps^{\N}=.us^{\N}=.uw^{.\N}$. Since $s$ is a minimal period of $\alpha$, $w=s^{k}$ for some $k\in\mathbb{N}$. Thus, $v=ps^{m_1+k}$ and for $m_2=m_1+k$ we get the result.
\end{proof}

Lemma \ref{rational} implies the following.

\begin{Corollary}\label{cor:rat}
Let $\alpha=.ps^\N$ be a rational number not in $\zz$ and assume that $s$ is a minimal period of $\alpha$. If $f\in F$ fixes $\alpha$ and has slope $2^a$ at $\alpha$, then $a$ is divisible by the length of $s$.
\end{Corollary}

\begin{proof}
By Lemma \ref{rational}, $f$ has a pair of branches of the form $ps^{m_1}\rightarrow ps^{m_2}$ for some $m_1,m_2\ge 0$. The slope of $f$ on the interval $[ps^{m_1}]$ is $2^{(m_1-m_2)|s|}$.
\end{proof}

\subsection{Natural copies of $F$}\label{sec:copies}

Let $f$ be a function in Thompson group $F$. The \emph{support of $f$}, denoted $\Supp(f)$, is the closure in $[0,1]$ of the subset $\{x\in(0,1):f(x)\neq x\}$. We say that $f$ \emph{has support in an interval $J$} if the support of $f$ is contained in $J$. Note that in this case the endpoints of $J$ are necessarily fixed by $f$. Hence the set of all functions from $F$ with support in $J$ is a subgroup of $F$. We denote this subgroup by $F_J$.

Let  $S$ be a subset of $[0,1]$, the notation $\Stab(S)$ will be used for the pointwise stabilizer of $S$ in $F$. Thus, if $f$ has support in a closed interval $[a,b]$, then $f\in \Stab([0,a]\cup[b,1])$. We note that if $f$ has support in an interval $J$ and $g\in F$, then $f^g$ has support in the interval $g(J)$. Similarly, $F_J^g=F_{g(J)}.$

Thompson group $F$ contains many copies of itself (see \cite{Brin}). The copies of $F$ we will be interested in will be of the following simple form. Let $a$ and $b$ be numbers from $\zz$ and consider the subgroup $F_{[a,b]}$.
We claim that $F_{[a,b]}$ is isomorphic to $F$.  Note that $F$ can be viewed as a subgroup of $\pl_2(\mathbb{R})$ of all piecewise linear homeomorphisms of  $\R$ with finite number of finite dyadic break points and absolute values of all slopes powers of 2.  Let $f\in \pl_2(\mathbb{R})$ be a function which maps $0$ to $a$ and $1$ to $b$, (such a function clearly exists). Then $F^f$ is the subgroup of $\pl_2(\mathbb{R})$  of all orientation preserving homeomorphisms with support in $[a,b]$, that is, $F^f=F_{[a,b]}$.

Let $u$ be a finite binary word and $[u]$ be the interval associated with it. The isomorphism between $F$ and $F_{[u]}$ can also be defined by  using tree-diagrams. Let $g$ be an element of $F$ represented by a tree-diagram $(T_+,T_-)$. We map $g$ to an element in $F_{[u]}$, denoted by $g_{[u]}$ and referred to as the \emph{copy of $g$ in $F_{[u]}$}.
To construct the element $g_{[u]}$ we start with a minimal finite binary tree $T$ which contains the branch $u$. We take two copies of the tree $T$. To the first copy, we attach the tree $T_+$ at the end of the branch $u$. In the second copy we attach the tree $T_-$ at the end of the branch $u$. The resulting trees are denoted by $R_+$ and $R_-$, respectively. The element $g_{[u]}$ is the one represented by the tree-diagram $(R_+,R_-)$. Note that if $g$ consists of pairs of branches $v_i\to w_i, i=1,...,k,$ and $B$ is the set of branches of $T$ which are not equal to $u$, then $g_{[u]}$ consists of pairs of branches $uv_i\to uw_i, i=1,...,k$, and $p\to p, p\in B$.

For example, the copies of the generators $x_0,x_1$ of $F$ in $F_{[0]}$ are depicted in Figure \ref{fig:0x0x1}.
It is obvious that these copies generate the subgroup $F_{[0]}$.

\begin{figure}[ht]
\centering
\begin{subfigure}{.5\textwidth}
  \centering
  \includegraphics[width=.5\linewidth]{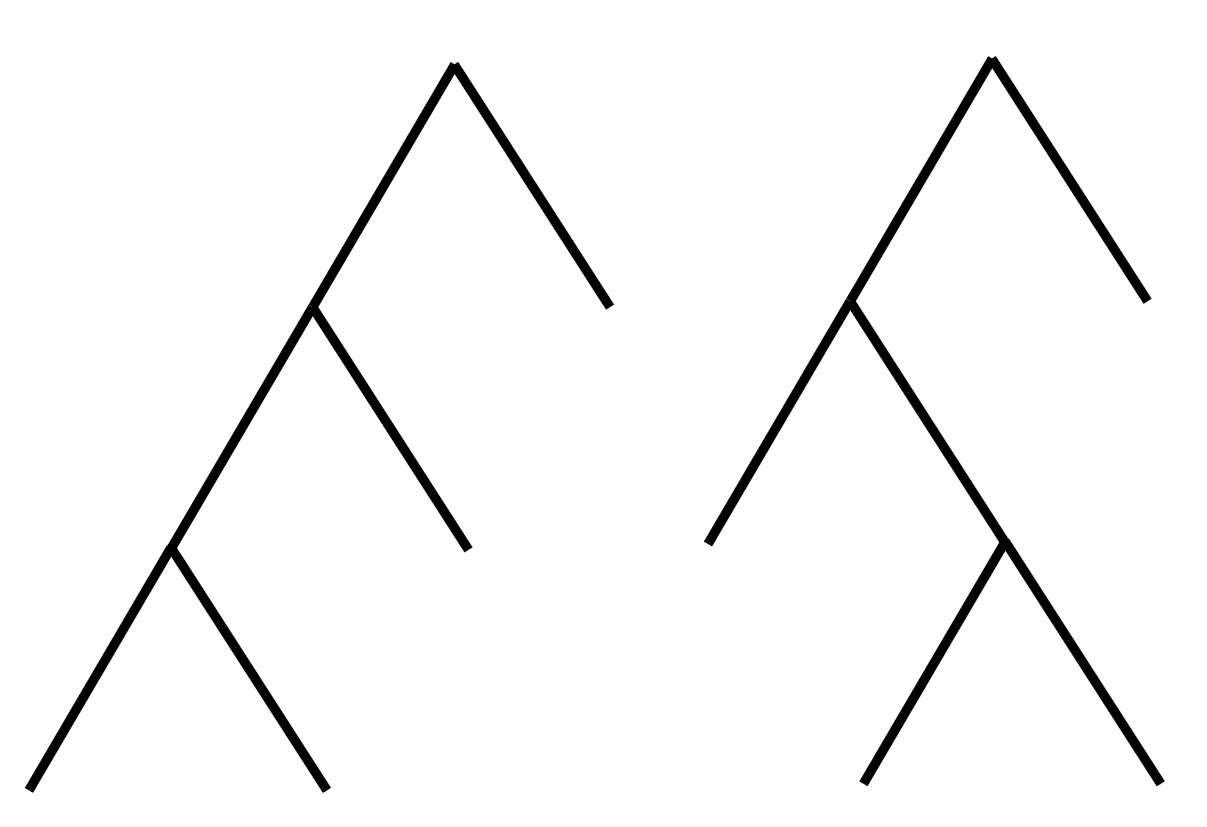}
  \caption{The tree-diagram of $(x_0)_{[0]}$}%The tree diagram $(T_+,T_-)$ of $x_0$}
  \label{fig:0x0}
\end{subfigure}%
\begin{subfigure}{.5\textwidth}
  \centering
  \includegraphics[width=.5\linewidth]{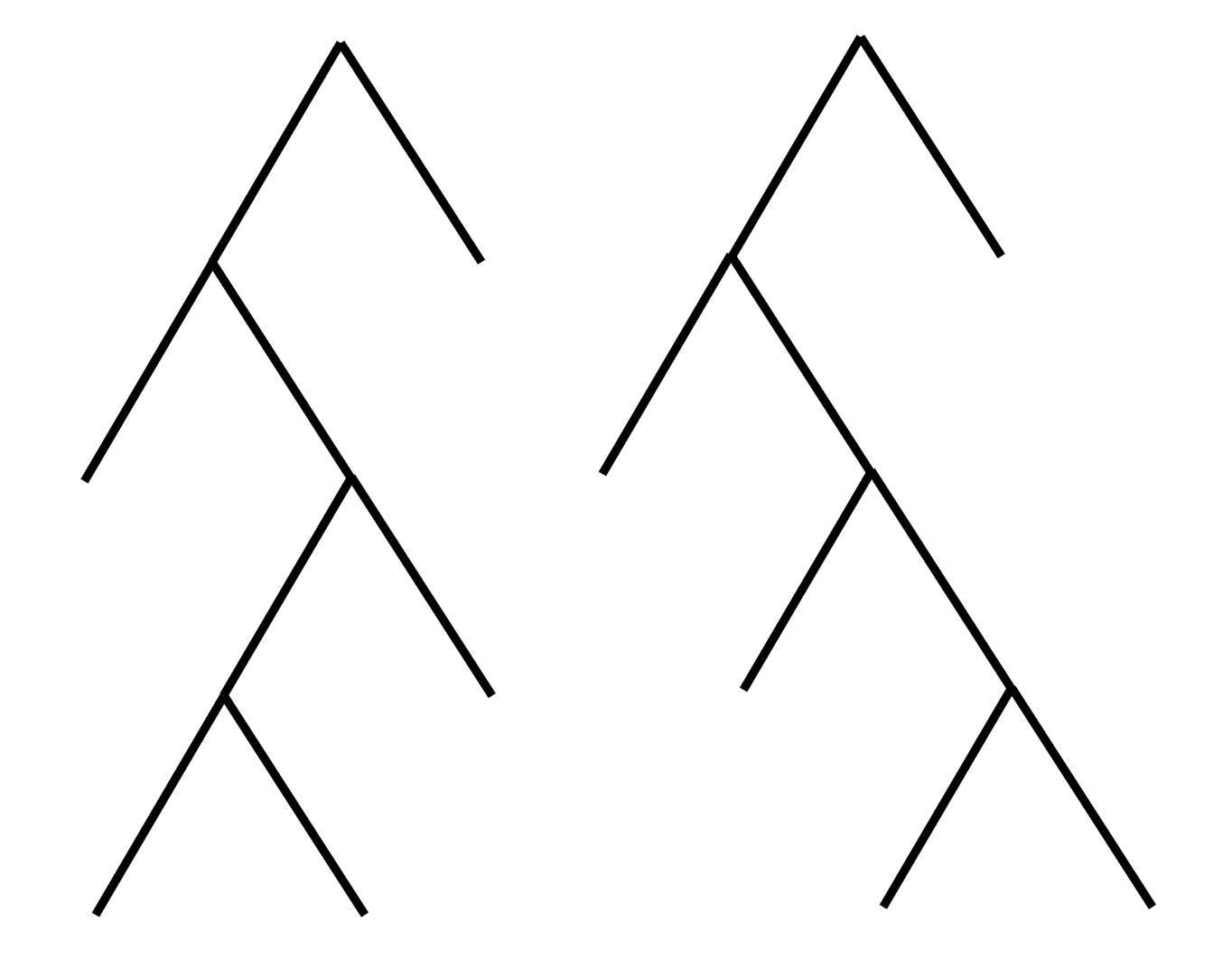}
  \caption{The tree-diagram of $(x_1)_{[0]}$}%The tree diagram $(T_+,T_-)$ of $x_1$}
  \label{fig:0x1}
\end{subfigure}
\caption{}%(a) The tree-diagram of $(x_0)_{[0]}$. (b) The tree-diagram of $(x_1)_{[0]}$.}
\label{fig:0x0x1}
\end{figure}

The isomorphism above guarantees that if $f,g\in F$ then $f_{[u]}g_{[u]}=(fg)_{[u]}$. Notice that in terms of branches, the copy $g_{[u]}$ of an element $g\in F$ can be characterized as the element in $F_{[u]}$ which for each pair of branches $v\rightarrow w$ of $g$ takes the branch $uv$ to the branch $uw$.

Using this isomorphism, we define an addition operation in Thompson group $F$ as follows. We denote by $\1$ the trivial element in $F$. We define the sum of an element $g\in F$ with the trivial element $\1$, denoted by $g\oplus\1$,
to be the copy of $g$ in $F_{[0]}$. Similarly, the sum of $\1$ and $g$, denoted by $\1\oplus g$, is the copy of $g$ in $F_{[1]}$. If $g,h\in F$ we define the \emph{sum} of $g$ and $h$, denoted by $g\oplus h$, to be the product $(g\oplus \1)(\1\oplus h)$. It is easy to see that for $g=\1$ or $h=\1$ this definition coincides with the previous one.
Note that $x_{1}=\1\oplus x_0$. In particular, $x_0$ and $\1\oplus x_0$ generate the whole $F$.
If we denote by $\zeta$ the function $t\mapsto 1-t$ from $\Homeo([0,1])$, then $F^\zeta=F$ and $(g\oplus h)^{\zeta}=h^{\zeta}\oplus g^{\zeta}$. One can check that $x_0^{\zeta}=x_0^{-1}$. Since $(x_0\oplus \1)^\zeta=\1\oplus x_0^{-1}=x_1^{-1}$ and $x_0^{-1},x_1^{-1}$ clearly generate $F$, we get that $x_0$ and $x_0\oplus \1$ generate $F$.
%Since $x_0\oplus 1=x_0^2x_1^{-1}x_0^{-1}$, $x_0$ and $x_0\oplus 1$ also generate $F$.

Note also that if $G$ is a subgroup of $F$, then the subgroup $\1\oplus G=\{\1\oplus g:g\in G\}$ is isomorphic to $G$. Similarly for $G\oplus \1$.

\section{The structure of stabilizers of finite sets}

It is known \cite{CFP}, that the derived subgroup of $F$ is exactly the subgroup $F_{(0,1)}$ of all functions with  support in $(0,1)$. Equivalently, $[F,F]$ is the subgroup of all functions with slope $1$ both at $0^+$ and at $1^-$.

\begin{Lemma}\label{der}
Let $a<b$ be any two numbers in $[0,1]$. Then the group $F_{(a,b)}$ of all functions with support in $(a,b)$ is isomorphic to the derived subgroup of $F$.
\end{Lemma}

\begin{proof}
We prove that $F_{(a,b)}$ is isomorphic to $F_{(0,1)}=[F,F]$.
Let $\{a_j\}_{j\in\N}$ and $\{b_j\}_{j\in \N}$ be sequences of numbers from $\zz$ such that
\begin{enumerate}
\item[(1)] $\{a_j\}_{j\in\N}$ is strictly decreasing and converges to $a$,
\item[(2)] $\{b_j\}_{j\in\N}$ is strictly increasing and converges to $b$; and
\item[(3)] $a_1<b_1$.
\end{enumerate}

Similarly, let $\{c_j\}_{j\in\N}$ and $\{d_j\}_{j\in \N}$ be sequences
of numbers from $\zz$ such that
\begin{enumerate}
\item[(1)] $\{c_j\}_{j\in\N}$ is strictly decreasing and converges to $0$,
\item[(2)]  $\{d_j\}_{j\in\N}$ is strictly increasing and converges to $1$; and
\item[(3)] $c_1<d_1$.
\end{enumerate}

Recall that $F_{[a_j,b_j]}$ is the subgroup of $F$ of all elements with support in $[a_j,b_j]$. Thus, $F_{(a,b)}$ is the increasing union of subgroups
$$F_{(a,b)}=\bigcup_{j\in\N}F_{[a_j,b_j]}$$ (each of these subgroups is a copy of $F$, see Section \ref{sec:copies}).
Similarly,
$$F_{(0,0)}=\bigcup_{j\in\N}F_{[c_j,d_j]}.$$
To prove the isomorphism between $F_{(a,b)}$ and $F_{(0,0)}$ it suffices to find a family of compatible isomorphisms $\psi_j\colon F_{[a_j,b_j]}\to F_{[c_j,d_j]}$, $j\in\mathbb{N}$. By \emph{compatible} we mean that for all $j>1$, the restriction of $\psi_j$ to $F_{[a_{j-1},b_{j-1}]}$ coincides with $\psi_{j-1}$.

We choose elements $g_j\in F$ inductively for $j\in\N$, such that
\begin{enumerate}
\item[(1)] for each $j$, $g_j(a_j)=c_j$ and $g_j(b_j)=d_j$; and
\item[(2)] for each $j>1$, the element $g_j$ coincides with the element $g_{j-1}$ on the interval $[a_{j-1},b_{j-1}]$.
\end{enumerate}
Such a choice is clearly possible by Remark \ref{r:choice}.

Notice that for each $j\in\N$, ${F_{[a_j,b_j]}}^{g_j}=F_{[c_j,d_j]}$.
Similarly, if $j>1$ and $h\in F_{[a_{j-1},b_{j-1}]}$, then $h^{g_j}=h^{g_{j-1}}$. Indeed, that follows from $h$ having support in $[a_{j-1},b_{j-1}]$ and condition (2) in the choice of $g_j$.

Thus, one can define a compatible family of isomorphisms  $\psi_j\colon F_{[a_j,b_j]}\to F_{[c_j,d_j]}$, by taking $\psi_j$ to be the isomorphism of conjugation by $g_j$. % for all $i$.
\end{proof}

\begin{Theorem}\label{thm:semi}
Let $U$ be a finite set of numbers in $(0,1)$. Assume that $U=U_1\cup U_2\cup U_3$ is the natural partition of $U$.
Let $m_i=|U_i|$, $i=1,2,3$, $r=|U|=m_1+m_2+m_3$.
Then $H_U$ is isomorphic to a semidirect product
$$H_U\cong [F,F]^{r+1}\rtimes \mathbb{Z}^{2m_1+m_2+2}.$$ Since $[F,F]$ is simple, the rank of the first integral homology group of $H_U$ is $2m_1+m_2+2$.
\end{Theorem}

\begin{proof}
For each $\alpha\in U_1$ we choose closed intervals $L_\alpha$ and $R_\alpha$ of positive length with endpoints in $\zz$ such that $\alpha$ is the right endpoint of $L_\alpha$ and left endpoint of $R_\alpha$. We can choose the intervals $L_\alpha$ and $R_\alpha$ to be small enough so that they do not contain points from $U_2\cup U_3$ and the interiors of all these intervals are pairwise disjoint.

For $\alpha\in U_1$, we choose elements $g_\alpha$ and $f_\alpha$ such that $g_\alpha$ has support in $L_\alpha$ and slope $2$ at $\alpha^-$ and
$f_\alpha$ has support in $R_\alpha$ and slope $2$ at $\alpha^+$.

For each $\beta\in U_2$
we have $\beta=.p_\beta s_\beta^\N$ for some finite binary words $p_\beta$ and $s_\beta$, where $s_\beta$ is a minimal period of $\beta$. Let $C_\beta$, for $\beta\in U_2$, be pairwise disjoint open intervals with endpoints in $\zz$, such that for all $\beta\in U_2$, $\beta\in C_\beta$.
Assume also that all $C_\beta$ are disjoint from the union of all $L_\alpha$ and $R_\alpha$, $\alpha\in U_1$.
%\cap R_n=\emptyset$ and that $\alpha_{n+m+1}\not\in C_{n+m}$.
For each $\beta\in U_2$ we choose an element $h_\beta$ such that $h_\beta$ has support in $C_\beta$ and has a pair of branches $p_\beta s_\beta^{k}\rightarrow p_\beta s_\beta^{k-1}$ for some $k\in\mathbb{N}$ (the number $k$ can be chosen independently from $\beta$). In particular, $h_\beta$ fixes $\beta$ and has slope $2^{|s_\beta|}$ at $\beta$.

Finally, we choose two additional elements, corresponding to the fixed points $0$ and $1$.
Let $N_0, N_1$ be closed intervals with disjoint interiors, containing 0 and 1 respectively and having endpoints in $\zz$. We assume that $N_0$ and $N_1$ are disjoint from all the intervals $L_\alpha, R_\alpha, C_\beta$ chosen above and do not contain any numbers from $U_3$.
We choose elements $f$ and $g$ such that  $f$ has support in $N_0$ and slope $2$ at $0^+$ and
 $g$ has support in $N_1$ and slope $2$ at $1^-$.

Note that the elements $g_\alpha, f_\alpha (\alpha\in U_1),g_\beta (\beta\in U_2), f, g$
belong to $H_U$. We let $G$ be the subgroup of $H_U$ generated by these elements. Since the interiors of supports of these elements are pairwise disjoint, they pairwise commute. Thus, $G$ is isomorphic to $\mathbb{Z}^{2n+m+2}$.

Let $\gamma_1,...,\gamma_r$ be the elements of $U$ in increasing order. Let $\gamma_0=0, \gamma_{r+1}=1$. Denote by $S$ the group of all elements of $F$ that fix open neighborhoods of each $\gamma_i, i=0,...,r+1$. Clearly $S\le H_U$.

We claim that $H_U$ is generated by $S$ and $G$.
Indeed, let $h\in H_U$. By Corollary \ref{cor:irr}, $h$ fixes an open neighborhood of each irrational number in $U$. We claim that one can multiply $h$ from the right by a suitable element $y\in G$ so that the slope of $hy$ at every point $\gamma_i, i=0,...,r+1$ would be $1$ (then obviously $hy\in S$ and so $h\in\la S\cup G\ra$).

Assume that the slope of $h$ at $0^+$ is $2^{\ell}$ for some $\ell$. Then $hf^{-\ell}$ has slope $1$ at $0^+$. Multiplying $h$ by $f^{-\ell}$ does not affect the slope at any point  $\gamma_j, j>0$. Thus, we can replace $h$ by $hf^{-\ell}$. Proceeding in this manner, one can make the slope at each point $\gamma_j$ be $1$ by multiplying from the right by elements of $G$ (we use $f_\alpha, g_\alpha$ for $\alpha\in U_1$, $g_\beta$ for $\beta\in U_2$ and $g$ for $j=r+1$).

To finish the proof we observe that $S$ is a normal subgroup of $H_U$ and so $H_U=SG$.
We claim that $S\cap G$ is trivial.  Indeed, the slopes of $h$ at (both sides) of the numbers from $U$ determine uniquely the element $y\in G$. Thus the only element of $G$ that fixes an open neighborhood around each $\gamma\in U$ is the identity.
Thus, $H_U=S\rtimes G$. It remains to note that the group $S$ is isomorphic to the direct product $$F_{(\gamma_0,\gamma_1)}\times F_{(\gamma_1,\gamma_2)}\times \ldots\times F_{(\gamma_r,\gamma_{r+1})}$$ and by Lemma \ref{der} is isomorphic to $[F,F]^{r+1}.$
\end{proof}

\begin{Corollary} \label{c:1} If $U$ and $V$ are finite sets of numbers from $(0,1)$ and $|U|\ne |V|$, then $H_U$ and $H_V$ are not isomorphic.
\end{Corollary}

\proof Indeed, by Theorem \ref{thm:semi} the derived subgroup of $H_U$ is isomorphic to the direct product of $|U|+1$ copies of the simple group $[F,F]$. Thus it has $2^{|U|+1}$ normal subgroups. So it cannot be isomorphic to a direct power of a different number of simple groups.
\endproof

The following is an immediate  corollary of the proof of Theorem \ref{thm:semi} (see \cite[Example A12.13]{Bieri} ).

\begin{Corollary} The R. Thompson group $F$ is a semidirect product of the derived subgroup $[F,F]$ and the Abelian subgroup generated by $x_0\oplus \1$ and $x_1=\1\oplus x_0$.
\end{Corollary}

\section{Isomorphism between stabilizers of finite sets}\label{sec:iso}

\subsection{Isomorphic stabilizers of finite sets}

Let $U=\{\gamma_1,\dots,\gamma_n\}$ be a set of numbers from $[0,1]$ (here and below we assume that $\gamma_1,\dots,\gamma_n$ are listed in increasing order). Let $U=U_1\cup U_2\cup U_3$ be the natural partition of $U$. Then we can define the {\em type} $\tau(U)$, to be a word  in the alphabet $\{1,2,3\}$ by taking the word $\gamma_1\gamma_2...\gamma_n$ and replacing each $\gamma_j\in U_i$ by the letter $i$. Note that by Corollary \ref{c:1} if $U,V\subseteq (0,1)$ and $|\tau(U)|\ne |\tau(V)|$, then $H_U$ and $H_V$ are not isomorphic.

\begin{Theorem} \label{th:m1} If $U$ and $V$ are two finite sets of numbers from $(0,1)$ and $\tau(U)\equiv \tau(V)$, then the subgroups $H_U$ and $H_V$ are isomorphic.
\end{Theorem}

To prove this theorem, we will realize $H_U$ and $H_V$ as iterated ascending HNN-extensions. Assuming $\tau(U)\equiv \tau(V)$, we will prove that the base groups of the HNN-extensions are isomorphic and the actions of the stable letters commute with the isomorphism between the relevant base groups. That will imply the result.

We will need the following three lemmas.

\begin{Lemma}\label{xy}
Let $a, b\in [0,1]$ be such that $a<b$. Let $x,y\in (a,b)\cap \zz$  be  such that $x<y$. Then $F_{[a,b]}$ is generated by $F_{[a,y]}$ and $F_{[x,b]}$.

%\la F_{[a,y]}\cup F_{[x,b]}\ra $$
\end{Lemma}

\begin{proof}
Let $f\in F_{[a,b]}$. If $f(x)< y$, then $f([a,x])\subset [a,y)$ and there exists a function $h\in F_{[a,y]}$ such that $h$ coincides with $f$ on the interval $[a,x]$. Then the function $fh^{-1}$ fixes the interval $[a,x]$. In particular, $fh^{-1}\in F_{[x,b]}$ and so $f\in F_{[x,b]}F_{[a,y]}$.
If $f(x)\ge y$, then $f(y)>y$. There is a function $g\in F_{[x,b]}$ such that $g(f(y))=y$. Then $y$ is a fixed point for $fg$, so $fg\in F_{[a,y]}F_{[y,b]} \subseteq F_{[a,y]}F_{[x,b]}$.
\end{proof}

\begin{Lemma}\label{0}
Let $a,b,c\in (0,1)$ be such that $a<b$ and $a<c$. Let $y\in (a,b)\cap(a,c)\cap \zz$. Then there exists an isomorphism $\psi\colon F_{[a,b)}\to F_{[a,c)}$ such that
\begin{enumerate}
\item[(1)] $\psi$ is the identity map on $F_{[a,y]}$, and
\item[(2)] For any $x\in (a,y)$, $\psi(F_{[x,b)})= F_{[x,c)}$.
\end{enumerate}
\end{Lemma}

\begin{proof}
We adapt the proof of Lemma \ref{der}. Let $\{b_j\}_{j\in\mathbb{N}}$ be an increasing sequence of numbers in $[a,b)\cap \zz$ which converges to $b$ and such that $b_1=y$. Let $\{c_j\}_{j\in\mathbb{N}}$ be an increasing sequence of numbers in $[a,c)\cap \zz$ which converges to $c$, and assume that $c_1=y$. To define an isomorphism
$$\psi\colon F_{[a,b)}=\bigcup F_{[a,b_j]}\to F_{[a,c)}=\bigcup F_{[a,c_j]}$$
we choose a sequence of elements $g_j$ in a similar way to that in Lemma \ref{der}. We let $g_1$ be an element which fixes the interval $[a,y]$. In particular, $g_1(b_1)=c_1$.
For each $j>1$, we let $g_j$ be an element such that $g_j(b_j)=c_j$ and such that $g_j|_{[a,b_{j-1}]}=g_{j-1}|_{[a,b_{j-1}]}$. The choice of elements $g_j$ defines compatible isomorphisms $\psi_j\colon F_{[a,b_j]}\to F_{[a,c_j]}$, where for each $j$, $\psi_j$ is the isomorphism of conjugation by $g_j$. The family of isomorphisms $\psi_j$ gives the required isomorphism $\psi$.

It suffices to prove that $\psi$ satisfies conditions (1) and (2).
If $h\in F_{[a,y]}=F_{[a,b_1]}$, then $\psi(h)=\psi_1(h)=h^{g_1}=h$, where the last equality follows from $g_1$ fixing the support of $h$. Thus, condition (1) holds. Let $x\in (a,y)$ and let $f\in F_{[x,b)}$. Then $f\in F_{[x,b_j]}$ for some $j\in\N$ and so $\psi(f)=f^{g_j}$. Notice that $\Supp(\psi(f))=g_j(\Supp(f))\subseteq g_j([x,b_j])=[x,c_j]\subseteq [x,c)$. Thus, $\psi(F_{[x,b)})\subseteq F_{[x,c)}$. Considering  $\psi^{-1}$ instead of $\psi$ gives the inverse inclusion.
\end{proof}

The proof of the following lemma is similar to the proof of Lemma \ref{0}.

\begin{Lemma}\label{1}
Let $a,b,c\in (0,1)$ be numbers such that $a<b$ and $c<b$. Let $y\in(a,b)\cap(c,b)\cap \zz$. Then there exists an isomorphism $\psi\colon F_{(a,b]}\to F_{(c,b]}$ such that
\begin{enumerate}
\item[(1)] $\psi$ is the identity map on $F_{[y,b]}$, and
%|_{F_{[y,b]}}=Id|_{F_{[y,b]}}$; and
\item[(2)] For any $x\in (y,b)$, $\psi(F_{(a,x]})= F_{(c,x]}$.
\end{enumerate}\qed
\end{Lemma}

\begin{proof}[Proof of Theorem \ref{th:m1}]
We start by replacing $H_V$ by a subgroup $H_W$ which is closer to $H_U$ in some sense.
Let $U=U_1\cup U_2\cup U_3$ be the natural partition.
For each $\beta\in U_2$  let $\beta=.p_\beta s_\beta^\N$ where $s_\beta$ is a minimal period. By replacing the prefixes $p_\beta$ by longer prefixes $p_\beta s_\beta^k$ if necessary, we can assume that the intervals $[p_\beta]$, $\beta\in U_2$ are pairwise disjoint and that each of these intervals contains exactly one element of $U\cup \{0,1\}$, the  number $\beta$.

Similarly, let $V=V_1\cup V_2\cup V_3$ be the natural partition of $V$.
For each $\gamma\in V_2$, let $\gamma=.q_{\gamma}u_{\gamma}^\N$, where $u_\gamma$ is a minimal period of $\gamma$. Assume as above, that the prefixes $q_\gamma$ are long enough, so that the intervals $[q_\gamma]$, $\gamma\in V_2$ are pairwise disjoint and contain exactly one number from $V\cup \{0,1\}$, the number $\gamma$. We claim that $H_V$ is isomorphic to a group $H_W$ such that $\tau(W)\equiv \tau(U)\equiv \tau(V)$ and $W=W_1\cup W_2\cup W_3$ (the natural partition of $W$) satisfies the following conditions.
\begin{enumerate}
\item[(1)] $W_1=U_1$.
\item[(2)] For each $\delta\in W_2$, $\delta=.p_\beta u_\gamma^\N$ where $\beta$ and $\gamma$ occupy the same position in the ordered sets $U$ and $V$ respectively as $\delta$ in $W$ (that is, the two  natural isomorphisms $\psi_{wu}\colon W\to U$, $\psi_{wv}\colon W \to V$  take $\delta$ to $\beta$ and $\delta$ to $\gamma$ respectively).
\end{enumerate}
Indeed, the conditions on the intervals $[p_j]$ and $[q_j]$ and Lemma \ref{choice} guarantee that there exists $f\in F$ such that
\begin{enumerate}
\item[(1)] for each $\gamma\in V_1$, $f(\gamma)\in U_1$; and
\item[(2)] for each $\gamma\in V_2$, $f$ has a pair of branches $q_\gamma\rightarrow p_{\psi_{vu}(\gamma)}$, where $\psi_{vu}$ is the natural isomorphism from $V$ to $U$.
\end{enumerate}
Conjugating $H_V$ by $f$ yields a group $H_W$ as described, where $W=f(V)$.

It will suffice to prove the isomorphism of $H_U$ and $H_W$.
We start by constructing isomorphic subgroups $K_U\le H_U$ and $K_W\le H_W$.

Let $I$ be the interval $[0,1]$. We remove from $I$ the endpoints $0$ and $1$, all the numbers in $U$ as well as the entire intervals $[p_\beta]$ for $\beta\in U_2$. The result is a set $J_U\subseteq [0,1]$. If $|U|=n$, then $J_U$ is a union of $n+1$ open intervals $(a_i,b_i)$, $i=1,\dots,n+1$ ordered from left to right.
 We define $K_U$ to be the subgroup generated by all $F_{[a_i,b_i]}$, that is, the direct product of these subgroups:
$$K_U=\prod_{i=1}^{n+1} F_{[a_i,b_i]}.$$
Clearly, $K_U$ fixes all points in $U$. In particular $K_U\le H_U$.

Similarly, removing from $I$ the endpoints $0$ and $1$, all numbers from $W$, as well as the intervals $[p_\beta]$ for $\beta\in U_2$ results in a union of $n+1$ open intervals $(c_i,d_i)$ for $i=1,\dots,n+1$.
We let
$$K_W=\prod_{i=1}^{n+1} F_{[c_i,d_i]}.$$

Note that $a_i, b_i, c_i, d_i$ are endpoints of the removed subintervals or points from $U_1\cup U_3\cup W_3$. Hence each $a_i,b_i, c_i, d_i$ is either in $\zz$ or an irrational number. Furthermore, if $a_i$ (resp. $b_i$) is in $\zz$ then $c_i=a_i$ (resp. $d_i=b_i$). If $a_i$ (resp. $b_i$) is irrational, then $c_i$ (resp. $d_i$)  is also irrational (this follows from the equality of types $\tau(U)$ and $\tau(W)$).
We shall be interested in an isomorphism from $K_U$ to $K_W$ with specific properties. For that, in any interval $(a_i,b_i)$ where exactly one of the endpoints is in $\zz$ (and hence exactly one of the endpoints of $(c_i,d_i)$ is in $\zz$ and coincides with the finite dyadic endpoint of $(a_i,b_i)$), we choose a number $y_i\in(a_i,b_i)\cap(c_i,d_i)\cap \zz$ (notice that the intersection cannot be empty since the intervals  $(a_i,b_i)$ and $(c_i,d_i)$ share an endpoint from $\zz$).

\begin{Lemma}\label{phi}
There exists an isomorphism $\phi\colon K_U\to K_W$  such that for each $i=1,\dots,n+1$, we have the following
\begin{enumerate}
\item $\phi(F_{[a_i,b_i]})= F_{[c_i,d_i]}$
\item If $a_i$ and $b_i$ are in $\zz$, then $\phi$ restricted to $F_{[a_i,b_i]}$ is the identity.
\item If $a_i$ is in $\zz$ and $b_i$ is irrational then $\phi$ restricted to $F_{[a_i,y_i]}$ is the identity and
for any $x\in (a_i,y_i)$, $\phi(F_{[x,b_i)})= {F_{[x,d_i)}}$.
\item If $a_i$ is irrational and $b_i$ is in $\zz$ then the restriction of $\phi$ to $F_{[y_i,b_i]}$ is the identity and
for any $x\in (y_i,b_i)$, $\phi(F_{(a_i,x]})= {F_{(c_i,x]}}$.
\end{enumerate}
\end{Lemma}

\begin{proof}
%The lemma follows from Lemmas \ref{0} and \ref{1}.
To define the isomorphism $\phi$ it suffices to define for each $i=1,\dots,n+1$ an isomorphism $\phi_{i}\colon F_{[a_i,b_i]}\to F_{[c_i,d_i]}$ and let $\phi=\phi_1\times \dots\times \phi_{n+1}$. That would guarantee condition (1).

If $a_i$ and $b_i$ are  both irrational, then $F_{[a_i,b_i]}=F_{(a_i,b_i)}$ and by Lemma \ref{der}, it is isomorphic to $F_{(c_i,d_i)}=F_{[c_i,d_i]}$. We let $\phi_i$ be any isomorphism between the groups.

If $a_i$ and $b_i$ are both in $\zz$, then $[a_i,b_i]=[c_i,d_i]$ and we take $\phi_i$ to be the identity automorphism of   $F_{[a_i,b_i]}$. This guarantees that $\phi$ will satisfy condition (2) of the lemma.

If $a_i$ is in $\zz$ and $b_i$ is irrational then $c_i=a_i$ and $d_i$ is irrational. Thus, $F_{[a_i,b_i]}=F_{[a_i,b_i)}$ and $F_{[c_i,d_i]}=F_{[a_i,d_i)}$. Since $y_i\in (a_i,b_i)\cap(a_i,d_i)\cap\zz$, we apply lemma \ref{0} to find an isomorphism $\phi_i$ so that condition (3) of the lemma would be satisfied.

Similarly, if $a_i$ is irrational and $b_i$ is in $\zz$ then $c_i$ is irrational and $d_i=b_i$. Thus, $F_{[a_i,b_i]}=F_{(a_i,b_i]}$ and $F_{[c_i,d_i]}=F_{(c_i,b_i]}$. Since $y_i\in (a_i,b_i)\cap(c_i,b_i)\cap\zz$, we apply Lemma \ref{1} to find an isomorphism $\phi_i$ so that condition (4) of the lemma would be satisfied.
\end{proof}

Next, we choose $m=|U_2|=|W_2|$ commuting elements $g_\beta\in H_U$ for $\beta\in U_2$ and $m$ commuting elements $f_\delta\in H_W$ for $\delta\in W_2$. To do so, we first choose a number from $\zz$ in each interval $[a_i,b_i]$ where at least one of the endpoints is in $\zz$. If $a_i$ is in $\zz$ and $b_i$ is irrational, we choose a number $x_i\in \zz$ in $(a_i,y_i)=(c_i,y_i)$. Similarly, if $a_i$ is irrational and $b_i$ is in $\zz$ we choose  a number $x_i\in \zz$ in $(y_i,b_i)=(y_i,d_i)$.
If $a_i$ and $b_i$ are both in $\zz$ we let $x_i=\frac{a_i+b_i}{2}$.

Recall that elements of $U_2$ are of the form $\beta=.p_\beta s_\beta^\N$,  and elements in $W_2$ are $\delta=.p_\delta u_\delta^\N\in W_2$.
Let, again, $\psi_{wu}$ be the order-preserving bijection $W\to U$. For every $\beta\in U_2$  and $\delta=\psi_{wu}\iv(\beta)$, we have $p_{\delta}=p_{\beta}$. We choose the element $g_\beta\in H_U$ and $f_\delta \in H_W$ as follows.
By construction, $\beta$ belongs to an interval $(b_i,a_{i+1})$ for some $i$ where $b_i, a_{i+1}$ are in $\zz$.  Let $x_i$ be the number from $\zz$ chosen in $[a_i,b_i]$ and let $x_{i+1}$ be the number from $\zz$ chosen in $[a_{i+1}, b_{i+1}]$.

We define $g_\beta$ and $f_\delta$ to be functions such that
\begin{enumerate}
\item[(1)] $g_\beta$ and $f_\delta$ have support in $[x_{i},x_{i+1}]$.
\item[(2)] $g_\beta$ has a pair of branches $p_\beta s_\beta\rightarrow p_\beta$ and $f_\delta$ has a pair of branches $p_\beta u_\delta \rightarrow p_\beta$.
\item[(3)] $g_\beta$ and $f_\delta$ coincide on the interval $[a_{i+1},x_{i+1}]=[c_{i+1},x_{i+1}]$ %and map it onto a sub-interval of itself.%
and map it linearly onto the the right half of itself.
%the interval $[a_{i_j+1},x_{i_j+1}]=[b_{i_j+1},x_{i_j+1}]$ linearly onto the right half of the interval.
\item[(4)] $g_\beta$ and $f_\delta$ coincide on the interval $[x_{i},b_{i}]=[x_{i},d_{i}]$ and map it linearly onto the left half of itself.

\end{enumerate}
By Lemma \ref{choice} one can indeed choose elements $g_\beta,f_\delta$ as described.
We note that since the interiors of the intervals $[x_i,x_{i+1}]$  are pairwise disjoint, the elements $g_\beta, \beta\in U_2,$ pairwise commute and the elements $f_\delta, \delta\in W_2,$ commute as well.
We let $G_U=\la g_\beta,\beta\in U_2\ra$ and $G_W=\la f_\delta, \delta\in W_2\ra$. Clearly, $G_U\cong G_W\cong \mathbb{Z}^m$.

\begin{Lemma}\label{generate}
$H_U$ is generated by $K_U$ and $G_U$. Similarly, $H_W$ is generated by $K_W$ and  $G_W$.
\end{Lemma}

\begin{proof}
Notice that the elements $g_\beta, \beta\in U_2,$ belong to $H_U$. Indeed, for each $\beta\in U_2$, $g_\beta$ has a pair of branches $p_\beta s_\beta \rightarrow p_\beta$. As such, it fixes $\beta=.p_\beta s_\beta^\N$. It also fixes all other numbers from $U$ because these numbers are not in the support of $g_\beta$. Thus, $\la K_U \cup G_U\ra \subseteq H_U$.

For the opposite inclusion, recall that $K_U=\prod_{i=1}^{n+1} F_{[a_i,b_i]}$ is the subgroup of $F$ of all functions which fix all points in $U$ as well as all intervals $[p_\beta]$ for $\beta\in U_2$.
Now let $h\in H_U$. By Lemma \ref{rational}, for each $\beta\in U_2$, the function $h$ has a pair of branches of the form $p_\beta s_\beta^{\ell_\beta}\rightarrow p_\beta s_\beta^{r_\beta}$ for some $\ell_\beta,r_\beta\ge 0$.
Let $\alpha$ be the smallest number in $U_2$ and consider
the element $h_\alpha=g_\alpha^{-\ell_\alpha}hg_\alpha^{r_\alpha}$. Since $g_\alpha^{-\ell_\alpha}$ has a pair of branches $p_\alpha\rightarrow p_\alpha s_\alpha^{\ell_\alpha}$, $h$ has a pair of branches $p_\alpha s_\alpha ^{\ell_\alpha}\rightarrow p_\alpha s_\alpha^{r_\alpha}$ and $g_\alpha^{r_\alpha}$ has a pair of branches $p_\alpha s_\alpha^{r_\alpha}\rightarrow p_\alpha$, the function $h_\alpha$ has a pair of branches $p_\alpha\rightarrow p_\alpha$. In other words, it fixes the interval $[p_\alpha]$. Notice that for all $\beta\in U_2$ such that $\beta\neq\alpha$, the support of $g_\alpha$ is disjoint from $[p_\beta]$. Thus, the
functions $h$ and $h_\alpha$ have the same pairs of branches $p_\beta s_\beta^{\ell_\beta}\rightarrow p_\beta s_\beta^{r_\beta}$ for $\beta\in U_2\setminus \{\alpha\}$. Considering other numbers from $U_2$ one by one, we find an element $ghg'$ where $g,g'\in G_U$  which
%$h_m=(\prod_{j=1}^m g_j^{\ell_j})^{-1} h (\prod_{j=1}^m g_j^{r_j})$
 fixes all intervals $[p_\beta]$ for $\beta\in U_2$. Since $ghg'$ also fixes all points in $U$, we have that $ghg'\in K_U$. Thus, $h\in G_UK_UG_U$.
\end{proof}

To finish the proof we prove that the group $H_U$ (resp. $H_W$) is an iterated ascending HNN-extension of the group $K_U$ (resp. $K_W$) with $m$-stable letters.

For this we are going to use the following simple and well known fact (see, for example, \cite[Lemma 2]{DruSap}).

\begin{Lemma} \label{l:hnn} Suppose that a group $G$ contains  a subgroup $K$ and an element $t$ such that
\begin{enumerate}
\item $G$ is generated by $K$ and $t$,
\item $t^n\not\in K$ for all $n>0$,
\item $K^t \le K$.
\end{enumerate}
Then $G$ is isomorphic to an ascending HNN-extension of $K$ with stable letter $t$.
\end{Lemma}

For each $\beta\in U_2$ we have
 $${K_U}^{g_\beta}\le K_U.$$
Indeed, $K_U$ is the group of all functions with support in $S=[a_1,b_1]\cup\dots\cup[a_{n+1},b_{n+1}]$. Since for every $\beta\in U_2$ $$\Supp(g_\beta)\cap S=[x_i,b_{i}]\cup[a_{i+1},x_{i+1}]$$
where $\beta\in(b_i,a_{i+1})$ and $g_\beta$ maps $[x_{i},b_{i}]\cup[a_{i+1},x_{i+1}]$ into itself, the support of each function in ${K_U}^{g_\beta}$ is inside $S$.

Let $U_2=\{\beta_1,...,\beta_m\}$. Since $g_\beta, \beta\in U_2,$ commute, for each $j=2,...,m$ we have
$$\la K_U,g_{\beta_1},\dots,g_{\beta_{j-1}}\ra^{g_{\beta_j}}\le \la K_U,g_{\beta_1},\dots,g_{\beta_{j-1}}\ra.$$
Since in addition $H_U$ is generated by $K_U$ and $g_\beta, \beta\in U_2$ and $g_{\beta_j}^k$ does not belong to
$\la K_U,g_{\beta_1},\dots,g_{\beta_{j-1}}\ra$ for every $k>0$ and every $j\ge 1$, the group $H_U$ is an iterated HNN-extension of $K_U$ with $m$ stable letters $g_\beta,\beta\in U_2$ by Lemma \ref{l:hnn}.
Similarly, $H_W$ is an iterated HNN-extension of $K_W$ with $m$ stable letters $f_{\delta_1},\dots,f_{\delta_m}$ where $W_2=\{\delta_1,\dots,\delta_m\}$.

Theorem \ref{th:m1} will follow if we prove that the actions of $g_{\beta_1},\dots,g_{\beta_m}$ on $K_U$ and the actions of $f_{\delta_1},\dots,f_{\delta_m}$ on $K_W$ commute with the isomorphism $\phi\colon K_U\to K_W$ from Lemma \ref{phi}.

We choose a generating set $Y$ of $K_U$ by fixing a generating set of $F_{[a_i,b_i]}$ for each $i=1,\dots,n+1$ and letting $Y$ be the union of these sets. If $a_i$ and $b_i$ are either both from $\zz$ or both irrational, we take the entire group $F_{[a_i,b_i]}$ as a generating set of itself. If $a_i$ is from $\zz$ and $b_i$ is irrational, recall that there are chosen  numbers $x_i<y_i$ in $(a_i,b_i)\cap \zz$. We take $F_{[a_i,y_i]}\cup F_{[x_i,b_i]}$ as a generating set of $F_{[a_i,b_i]}$ (notice that by Lemma \ref{xy}, this union generates $F_{[a_i,b_i]}$). Similarly, if $a_i$ is irrational and $b_i$ is from $\zz$, then there are chosen  numbers $y_i<x_i$ from $\zz$ in $(a_i,b_i)$. We take $F_{[a_i,x_i]}\cup F_{[y_i,b_i]}$ as a generating set of $F_{[a_i,b_i]}$.

The following lemma completes the proof of Theorem \ref{th:m1}.

\begin{Lemma}
For each $h\in Y$ and $\beta\in U_2$ we have
$$\phi(h^{g_\beta})=\phi(h)^{f_\delta}$$ where $\delta=\psi_{wu}\iv(\beta)$.
\end{Lemma}

\begin{proof}
By the construction of $Y$, $h\in F_{[a_i,b_i]}$ for some $i$. Suppose that the supports of $g_\beta$ and $h$ are disjoint. Then $h^{g_\beta}=h$. By Condition (1) in Lemma \ref{phi}, $\phi(h)\in F_{[c_i,d_i]}$. Since the support of $f_\delta$ is  disjoint from $[c_i,d_i]$, we get that $\phi(h)^{f_\delta}=\phi(h)=\phi(h^{{g_\beta}})$ as required.

Thus we can assume that the support of $g_{\beta}$ and $h$ are not disjoint. This implies that $\beta$ belongs to either the interval $[b_{i-1}, a_i]$ or to the interval $[b_i, a_{i+1}]$. We only consider the first case, the other case being similar.

Consider the interval $(a_{i},b_{i})$. The number  $a_{i}$ is from $\zz$. If $b_{i}$ is also from $\zz$, then by Condition (2) of Lemma \ref{phi}, the restriction of $\phi$ to $F_{[a_{i},b_{i}]}$ is the identity. In particular $\phi(h)=h$. Notice that $g_\beta$ maps the interval $[a_{i},b_{i}]$ into a sub-interval of itself. Thus, $h^{{g_\beta}}$ also belongs to $F_{[a_{i},b_{i}]}$. Therefore $\phi(h^{{g_\beta}})=h^{{g_\beta}}$. Thus it suffices to prove that $h^{g_\beta}=h^{f_\delta}$. That follows immediately from Conditions (1) and (3) in the definition of $g_\beta$ and $f_\delta$. Indeed, these conditions  imply that $g_\beta$ and $f_\delta$ coincide on the support of $h$ .

Next, we assume that $b_{i}$ is irrational. Since $h\in  F_{[a_{i},b_{i}]}\cap Y$, either $h\in F_{[a_{i},y_{i}]}$ or $h\in F_{[x_{i},b_{i}]}$.
In the first case, by Condition (3) in Lemma \ref{phi}, we have $\phi(h)=h$. Similarly, since $g_\beta$ maps the interval $[a_{i},y_{i}]$ onto a sub-interval of itself, the conjugate $h^{g_\beta}$ also belongs to  $F_{[a_{i},y_{i}]}$. As such $\phi(h^{{g_\beta}})=h^{g_\beta}$ and it suffices to prove that $h^{g_\beta}=h^{f_\delta}$. That follows as before from the fact that $g_\beta$ and $f_\delta$ coincide on the support of $h$. If $h\in F_{[x_{i},b_{i}]}=F_{[x_{i},b_{i})}$, then $g_\beta$ commutes with $h$ (indeed, their supports have disjoint interiors). By condition (3) from Lemma \ref{phi}, $\phi(h)\in F_{[x_{i},c_{i})}$, which implies that $\phi(h)$ commutes with $f_\delta$. As such $\phi(h^{g_\beta})=\phi(h)=\phi(h)^{f_\delta}$.
\end{proof} That completes the proof of Theorem \ref{th:m1}.
\end{proof}

\subsection{Algebraic structure of stabilizers of finite sets}\label{sec:alg}

The proof of Theorem \ref{th:m1} gives us more explicit information about the structure of the stabilizer $H_U$ for a finite set $U$. If $U=\{\alpha\}$ for $\alpha\in(0,1)$ such that $\alpha\notin U_2$, then
$$H_U=F_{[0,\alpha]}F_{[\alpha,1]}.$$
Indeed, this is obvious for $\alpha\in U_1$. If $\alpha\in U_3$ then every element of $H_U$ fixes an open neighborhood of $\alpha$ and thus belongs to the direct product $F_{[0,\alpha]}F_{[\alpha,1]}=F_{[0,\alpha)}F_{(\alpha,1]}$. Similarly, let $U=\{\alpha_1,\dots,\alpha_n\}\subseteq(0,1)$. Let $U=U_1\cup U_2\cup U_3$ be the natural partition. If $U_2=\emptyset$, then
$$H_U=F_{[0,\alpha_1]}\times F_{[\alpha_1,\alpha_2]}\times\dots\times F_{[\alpha_{n-1},\alpha_n]}\times F_{[\alpha_n,1]}.$$
If $U_2\neq\emptyset$ and $|U_1\cup U_3|=k$, then $U_1\cup U_3$ separates $U_2$ into a union of $k+1$ disjoint subsets
$U_{2,1},\dots,U_{2,k+1}$ (some of which might be empty). Let $i_1,\dots,i_{k}$ be the indexes such that $\alpha_{i_j}\in U_1\cup U_3$. Since $H_U=H_{U_2}\cap H_{U_1\cup U_3}$, the above equation (for the case $U_2=\emptyset$) implies that
$$\begin{array}{lr} H_U=(F_{[0,\alpha_{i_1}]}\cap H_{U_{2,1}})
\times (F_{[\alpha_{i_1},\alpha_{i_2}]}\cap H_{U_{2,2}})& \times\dots\times (F_{[\alpha_{i_{k-1}},\alpha_{i_{k}}]}\cap H_{U_{2,k}}) \\ & \times (F_{[\alpha_{i_k},1]} \cap H_{U_{2,k+1}}).
\end{array}
$$
For a set of numbers $V=\{\beta_1,\dots,\beta_m\}$ in $[0,1]$, if $\beta_1,\beta_m\notin U_2$, we let
$$B_V=F_{[\beta_1,\beta_m]}\cap H_{V\setminus\{\beta_1,\beta_m\}}.$$
Theorem \ref{th:m1} implies the following.

\begin{Lemma}\label{ib}
Let $U=\{\alpha_1,\dots,\alpha_n\}$ and $V=\{\beta_1,\dots,\beta_n\}$ be sets of numbers from $[0,1]$. Suppose that $\tau(U)\equiv \tau(V)$ and $\alpha_1,\alpha_n\not\in U_2$. Then the groups $B_U$ and $B_V$ are isomorphic.
\end{Lemma}

\begin{proof}
It follows from the proof of Theorem \ref{th:m1}. Indeed, using conjugation by an orientation preserving element of $\pl_2(\mathbb{R})$, one can assume that $U$ and $V$ are sets of numbers in $(0,1)$.
By Theorem \ref{th:m1}, we have that
$$H_U=F_{[0,\alpha_1]}\times B_U\times F_{[\alpha_n,1]}\ \ \hbox{ and }\ \ H_V=F_{[0,\beta_1]}\times B_V\times F_{[\beta_n,1]}.$$
are isomorphic. The isomorphism constructed in the proof of Theorem \ref{th:m1} maps $F_{[0,\alpha_1]}$ onto $F_{[0,\beta_1]}$, $F_{[\alpha_n,1]}$ onto $F_{[\beta_n,1]}$ and $B_U$ onto $B_V$.
\end{proof}

Let $U=\{\alpha_1,\dots,\alpha_n\}\subseteq[0,1]$ such that $\alpha_1,\alpha_n\notin U_2$. By Lemma \ref{ib}, the isomorphism class of $B_U$ depends only on $\tau(U)$. Thus, if $w\equiv \tau(U)$, we will use the notation $B_w$ for the group $B_U$.
If $\alpha_1,\alpha_n\in U_1$, then Lemma \ref{ib} implies that $B_w\cong B_{U\setminus\{\alpha_1,\alpha_n\}\cup\{0,1\}}=H_{U\setminus\{\alpha_1,\alpha_n\}}$.

\begin{Lemma}\label{back}
Let $U=\{\alpha_1,\dots,\alpha_n\}$ and $V=\{\beta_1,\dots,\beta_n\}$ be finite sets of numbers in $[0,1]$, such that $\tau(U)$ does not start or end with $2$. Assume that the word $\tau(U)$ is equal to $\tau(V)$ read backwards.
Then $B_U$ is isomorphic to $B_V$.
\end{Lemma}

\begin{proof}
If one conjugates $B_V$, where $V=\{\beta_1,\dots,\beta_n\}$ by the function $\zeta(t)=1-t$ in $\Homeo([0,1])$, one gets $B_W$ where $W=\{1-\beta_n,\dots,1-\beta_1\}\subseteq[0,1]$. Since $\tau(U)\equiv\tau(W)$, the result follows from Lemma \ref{ib}.
\end{proof}

Now let $U=\{\alpha_1,\dots,\alpha_{n}\}$ be a finite set of numbers in $(0,1)$. Let $\alpha_0=0$, $\alpha_{n+1}=1$ and $U'=U\cup\{\alpha_0,\alpha_{n+1}\}$. Let $w\equiv \tau(U')$. Let us represent $w$ as the product $1u_1i_1u_2i_2...u_{k+1}1$ where $i_j\in \{1,3\}, j=1,...,k$ and each word $u_i$ contains only letter 2.
The discussion above and Lemma \ref{ib} clearly imply that
$$H_U\cong B_{1u_1i_1}\times B_{i_1u_2i_2}\times\dots\times B_{i_{k}u_{k+1}1}.$$
Note that by Lemma \ref{back}, to completely describe the structure of $H_U$, we only need to describe groups $B_w$ where $w$ is a word of one of the following 6 kinds:
$11$, $33$, $13$, $12^m1$, $32^m3$ and $12^m3$ for $m\in\mathbb{N}$.
We note that $B_{11}$ is isomorphic to $F$, $B_{33}$ is isomorphic to the derived subgroup of $F$ (see Lemma \ref{der}) and that $B_{13}$ is isomorphic to the normal subgroup $\mathcal L$ of $F$ of all functions with slope $1$ at $1$ (indeed, $B_{13}\cong F_{[0,\frac{\pi}{4})}\cong F_{[0,1)}=\mathcal L$ by Lemma \ref{0}).
The group $B_w$ is more difficult to describe in the remaining $3$ cases. We leave the cases $w\equiv 32^m3$ and $w\equiv 12^m3$ to the reader. For $w\equiv 12^m1, m\in\mathbb{N}$, the proof of Theorem \ref{th:m1} shows the following.

\begin{Lemma}\label{HNN-ext}
Let $U$ be a set of $m$ rational numbers  in $(0,1)\setminus \zz$. Then $H_U\cong B_{12^m1}$ is isomorphic to an iterated ascending HNN-extension of $K=F^{m+1}$ with $m$ stable letters $t_1,\dots,t_m$ such that $t_j$ commutes with $t_i$ and $K^{t_j}< K$, so each $t_j$ corresponds to an endomorphism of $K$. That  endomorphism $\phi_j$ is defined as follows.
We denote by $\iota_i$ the injection of $F$ into the $i$th direct summand of $F^{m+1}$. Then
$$\begin{array}{ll} \phi_j\colon
& \iota_i(x_0)\to \iota_i(x_0) \hbox{ for all } i\notin\{j,j+1\}, \\
& \iota_i(x_1)\to \iota_i(x_1) \hbox{ for all } i\notin\{j,j+1\}, \\
& \iota_{j}(x_0)\to\iota_j(x_1x_0^{-1}) \\
& \iota_{j}(x_1)\to\iota_j(x_1) \\
& \iota_{j+1}(x_0)\to\iota_{j+1}(x_0x_1^{-1}) \\
& \iota_{j+1}(x_1)\to\iota_{j+1}(x_1^2x_2^{-1}x_1^{-1}).
\end{array}$$
\end{Lemma}

\begin{proof}
This follows from an analysis of the proof of Theorem \ref{th:m1} in the case when $U=U_2$. In the notations of the proof, the intervals $[a_i,b_i]$ for $i=1,\dots,m+1$ all have endpoints from $\zz$. Thus, the group $K_U=\prod_{i=1}^{m+1}F_{[a_i,b_i]}$ is isomorphic to the direct sum of $m+1$ copies of $F$. Each of the generators $g_{\beta_j}$ fixes the intervals ${[a_i,b_i]}$ for all $i\neq\{j,j+1\}$ and thus acts trivially on $F_{[a_i,b_i]}$. In addition, $g_{\beta_j}$ fixes the left half of $[a_j,b_j]$ and maps the right half of $[a_j,b_j]$ to its own left half. Similarly, $g_{\beta_j}$ fixes the right half of $[a_{j+1},b_{j+1}]$ and maps the left half of $[a_{j+1},b_{j+1}]$ to its right half. Using these facts and the isomorphism of $F_{[a_j,b_j]}$ with $F$ and $F_{[a_{j+1},b_{j+1}]}$ with $F$ one gets the above description of the automorphism $\phi_j$.
\end{proof}

Notice that Lemma \ref{HNN-ext} implies that if $U=U_2$, then $H_U$ has a generating set with $3m+2$ generators. In the following section we will improve this result and find the minimal number of generators.

\section{Stabilizers of finite sets of rational numbers}

Let $U\subset [0,1]$ be a finite set of rational numbers. In this section,  we are going to show that the rank of the first homology group of $H_U$ with integral coefficients (given by Theorem \ref{thm:semi}) coincides with the smallest number of generators of $H_U$.

We first consider the case where $U=U_2$. Recall (Section \ref{sec:copies}), that if $f\in F$ and $u$ is a finite binary word, then $f_{[u]}$ denotes the copy of $f$ in $F_{[u]}$. Recall also  that if $f$ consists of pairs of branches $v_i\to w_i$, then $f_{[u]}$ consists of pairs of branches $uv_i\to uw_i$ and some pairs of branches $p\to p$.  We will need the following three lemmas.

\begin{Lemma}\label{left}
Let $u$ be a finite binary word and let $g\in F$ be an element with a pair of branches  $u\to  u0$. Let $f\in F$ and consider the copy $f_{[u]}$ of $f$ in $F_{[u]}$. Then ${f_{[u]}}^g=(f\oplus \1)_{[u]}$. That is, conjugating the copy of $f$ in $F_{[u]}$ by $g$ gives the copy of $f\oplus \1$ in $F_{[u]}$.
\end{Lemma}

\begin{proof}
First, notice that $(f\oplus \1)_{[u]}=f_{[u0]}$. To prove that ${f_{[u]}}^g$ is equal to the copy of $f$ in $F_{[u0]}$ we will show that ${f_{[u]}}^g$ has support in the interval $[u0]$ and that for any pair of branches $v\to w$ of $f$, the element ${f_{[u]}}^g$ takes the branch $u0v$ to the branch $u0w$.

Since $f_{[u]}$ has support in $[u]$ and $g$ maps the interval $[u]$ onto $[u0]$, the conjugate
${f_{[u]}}^g$ has support in the interval $[u0]$. For any pair of branches $v\rightarrow w$ of $f$, the copy $f_{[u]}$ has a pair of branches $uv\rightarrow uw$. The pair of branches $u\rightarrow u0$ of $g$ imply that ${f_{[u]}}^g$ has a pair of branches $u0v\rightarrow u0w$ for any such $v$ and $w$. Indeed, $g^{-1}$ takes $u0v$ to $uv$, $f_{[u]}$ takes $uv$ to $uw$ and $g$ takes $uw$ to $u0w$.
\end{proof}

Similarly, we have the following.

\begin{Lemma}\label{right}
Let $u$ be a finite binary word and let $g\in F$ be an element with a pair of branches $u\to  u1$. Let $f\in F$ and consider the copy $f_{[u]}$ of $f$ in $F_{[u]}$. Then ${f_{[u]}}^g=(\1\oplus f)_{[u]}$. That is, conjugating the copy of $f$ in $F_{[u]}$ by $g$ gives the copy of $\1\oplus f$ in $F_{[u]}$.\qed
\end{Lemma}

\begin{Lemma}\label{x0}
Let $u$ be a finite binary word. Let $h_{\ell}$ and $h_r$ be functions from $F$ with supports disjoint from $[u]$. Let $g_{\ell}$ be a function in $F$ which takes the branch $u$ to $u0$ and fixes all points in the support of $h_{\ell}$. Similarly, let $g_r\in F$ be a function which takes the branch $u$ to $u1$ and fixes all points in the support of $h_r$.
Let $f_{\ell}=h_{\ell}(x_0)_{[u]}$ and $f_r=h_r(x_0)_{[u]}$. Then
$ (x_0)_{[u]}$ belongs to the subgroup $\la  f_{\ell},f_r, g_{\ell},g_r  \ra  .$
\end{Lemma}

\begin{proof}
Let $G=\la  f_{\ell},f_r, g_{\ell},g_r  \ra $.
We consider the conjugate of $f_{\ell}$ by $g_{\ell}$.
Since $g_{\ell}$ fixes all points in the support of $h_{\ell}$, the functions $g_{\ell}$ and $h_{\ell}$ commute. By Lemma \ref{left} ${(x_0)_{[u]}}^{g_{\ell}}=(x_0\oplus \1)_{[u]}$. Thus, ${f_{\ell}}^{g_{\ell}}=h_{\ell}(x_0\oplus \1)_{[u]}\in G$.
Similarly, using Lemma \ref{right} we get that ${f_r}^{g_r}=h_{r}(\1\oplus x_0)_{[u]}=h_{r}(x_1)_{[u]}$.
Note that $x_0\oplus \1=x_0^2x_1^{-1}x_0^{-1}$. Since $[u]$ is disjoint from the support of $h_r$, $(x_0)_{[u]}$ and $(x_1)_{[u]}$ commute with $h_r$.
Thus we have that ${f_r}^2({f_r}^{g_r})^{-1}f_r^{-1}=(x_0^2x_1^{-1}x_0^{-1})_{[u]}=(x_0\oplus \1)_{[u]}\in G$. That implies that $h_{\ell}\in G$ and, in turn, that $(x_0)_{[u]}\in G$.
\end{proof}

\begin{Theorem}\label{thm:rational}
Let $U=U_2$. Then $H_U$ has a generating set with $|U|+2$ elements (in fact a generating set of this size can be explicitly described).
\end{Theorem}

\begin{proof}Let $U=U_2=\{\alpha_1,...,\alpha_m\}$.
We start as in the proof of Theorem \ref{th:m1} by constructing a subgroup $K_U$ and elements $g_1,\dots,g_m$ such that $K_U$ and $g_1,\dots,g_m$ generate $H_U$. As above, $K_U$ will be the direct product of groups of the form $F_{[a_i,b_i]}$. It will be convenient to assume that $[a_i,b_i]$ are dyadic intervals.

For $j=1,\dots,m$, let $\alpha_j=.p_js_j^\N$ where $s_j$ is a minimal period. We assume that the prefixes $p_j$ are long enough so that the intervals $[p_1],\dots,[p_m]$ are pairwise disjoint and such that $0\notin [p_1]$ and $1\notin [p_m]$.

Let $T$ be a finite binary tree with $2m+1$ leaves represented by binary words $v_1\dots,v_{2m+1}$. The $m$ intervals associated with the even numbered branches; that is, $[v_2],[v_4],\dots,[v_{2m}]$, are pairwise disjoint, $0\notin [v_2]$ and $1\notin [v_{2m}]$.
Thus by Lemma \ref{choice} there exists an element $f\in F$ with pairs of branches $p_j\rightarrow v_{2j}$ for $j=1,\dots,m$.
Conjugating $H_U$ by $f$ results in a group $H_V$ where $V=\{\beta_1,\dots,\beta_m\}$ and for each $j=1,\dots,m$, $\beta_m=v_{2j}s_j^\N$. Clearly, it suffices to prove that $H_V$ is $m+2$ generated. For the rest of the proof, we rename the set $V$ and its elements by $U$ and $\alpha_1,\dots,\alpha_m$, respectively.

Notice that if one removes from $I=[0,1]$ the endpoints $0,1$ and the intervals $[v_2],\dots,[v_{2m}]$, one remains with a union of $m+1$ open intervals $(a_i,b_i)$ for $i=1,\dots,m+1$. It is obvious, that $[a_i,b_i]=[v_{2i-1}]$. Thus, we define the subgroup
$$K_U=\prod_{i=1}^{m+1}F_{[v_{2i-1}]}.$$

Next, we choose elements $g_1,\dots,g_m\in H_U$. Unlike in the proof of Theorem \ref{th:m1}, we do not require that the elements commute.  For each $j=1,\dots,m$ we let $g_j$ be an element with the following pairs of branches.
\[
  \begin{cases}
	v_{2j-1} &\rightarrow v_{2j-1}0\\
	v_{2j}s_j &\rightarrow v_{2j}\\
	v_{2j+1} &\rightarrow v_{2j+1}1\\
	v_k &\rightarrow v_k, \mbox{ for all $k\in\{1,\dots,2m+1\}\setminus\{2j-1,2j,2j+1\}$ }
  \end{cases}
\]
Note that the existence of an element $g_j$ with the required pairs of branches follows from Lemma \ref{choice}.

\begin{Lemma}\label{KU}
The group $H_U$ is generated by the subgroup $K_U$ and the elements $g_1,\dots,g_m$.
\end{Lemma}

\begin{proof}
The proof is identical to the proof of Lemma \ref{generate}. Indeed, the only conditions in the definition of the elements $g_1,\dots,g_m$ necessary for the proof are that $g_j$ has a pair of branches $v_{2j}s_j\rightarrow v_{2j}$ and that the support of $g_j$ is disjoint from the interval $[v_{2k}]$ for all $k\neq j$. Then given an element $h\in H_U$, one can multiply $h$ from the left and from the right by powers of $g_1,\dots,g_m$ to get an element $h'\in K_U$.
\end{proof}

\begin{Lemma}\label{odd}
The group $H_U$ is generated by the set
 $$S=\{(x_0)_{[v_{1}]},(x_0)_{[v_{3}]},\dots,(x_0)_{[v_{2m-1}]},(x_0)_{[v_{2m+1}]},g_1,\dots,g_m\}.$$
\end{Lemma}

\begin{proof}
By Lemma \ref{KU} and the definition of $K_U$, it suffices to prove that for all $j=1,\dots,m+1$, the subgroup $\la S \ra$ contains the subgroup $F_{[v_{2j-1}]}$.
For each $j=1,\dots,m$, the element $g_j$ has a pair of branches $v_{2j-1}\rightarrow v_{2j-1}0$. Thus, by Lemma \ref{left}, $${(x_0)_{[v_{2j-1}]}}^{g_j}=(x_0\oplus \1)_{[v_{2j-1}]}.$$ Since $(x_0)_{[v_{2j-1}]}$ and $(x_0\oplus \1)_{[v_{2j-1}]}$ generate $F_{[v_{2j-1}]}$, we have the inclusion $F_{[v_{2j-1}]}\subseteq\la S\ra$.
For $j=m+1$, we note that $g_{m}$ has a pair of branches $v_{2m+1}\rightarrow v_{2m+1}1$. Thus, by Lemma \ref{right},
$${(x_0)_{[v_{2m+1}]}}^{g_m}=(\1\oplus x_0)_{[v_{2m+1}]}=(x_1)_{[v_{2m+1}]}.$$
Since $(x_0)_{[v_{2m+1}]}$ and $(x_1)_{[v_{2m+1}]}$ generate $F_{[v_{2m+1}]}$, it is also contained in $\la S\ra$.
\end{proof}

To prove that $H_U$ is $(m+2)$-generated, we choose two elements $x$ and $y$ in $K_U$ as follows.
If $m$ is odd, we let
$$x=\prod_{i=1}^{\mathclap{\frac{m+1}{2}}} (x_0)_{[v_{4i-3}]}\ \ \  \mbox{ and }\ \ \  y=\prod_{i=1}^{\mathclap{\frac{m+1}{2}}} (x_0)_{[v_{4i-1}]}.$$
If $m$ is even we let
$$x=\left[\prod_{i=1}^{\frac{m}{2}} (x_0)_{[v_{4i-3}]}\right](x_0)_{[v_{2m-1}]}\ \ \  \mbox{  and     }\ \ \ y=\left[\prod_{i=1}^{\mathclap{\frac{m}{2}}} (x_0)_{[v_{4i-1}]}\right](x_0)_{[v_{2m+1}]}.$$
Notice that in both cases, all elements appearing in the product defining $x$ have disjoint support, thus all elements in the product commute. The same is true for the product defining $y$.

We claim that $H_U$ is generated by $x,y,g_1,\dots g_m$. By Lemma \ref{odd}, it suffices to prove the following.

\begin{Lemma}
The group $H=\la x,y,g_1,\dots,g_m\ra$ contains the elements $(x_0)_{[v_{2j-1}]}$ for $j=1,\dots,m+1$.
\end{Lemma}

\begin{proof}
We first consider the case where $m$ is odd.
Recall that in that case, $$x=\prod_{i=1}^{\mathclap{\frac{m+1}{2}}} (x_0)_{[v_{4i-3}]}.$$
For each $k\in\{2,\dots,\frac{m+1}{2}\}$ we apply Lemma \ref{x0} with $u\equiv v_{4k-3}$,
$$h_{\ell}=h_r=\prod_{{\substack{i=1 \\ i\neq k}}}^{{\frac{m+1}{2}}} (x_0)_{[v_{4i-3}]},$$
$g_{\ell}=g_{2k-1}$ and $g_r=g_{2k-2}$.
Note, that all the conditions of Lemma \ref{x0} are satisfied. Indeed, the support of $h_{\ell}=h_r$ is disjoint from $[u]=[v_{4k-3}]$, $g_{\ell}=g_{2k-1}$ fixes all intervals $[v_i]$ for $i\notin \{4k-3,4k-2,4k-1\}$ and in particular fixes the support of $h_{\ell}$. Similarly, $g_r$ fixes the support of $h_r$. By construction, $g_{\ell}=g_{2k-1}$ takes $u\equiv v_{4k-3}$ to $u0$ and $g_r$ takes $u$ to $u1$. Note that $f_{\ell}$ and $f_r$ from the lemma are both equal to $x$. Thus, by lemma \ref{x0},
 $$(x_0)_{[v_{4k-3}]}\in \la f_{\ell},f_{r},g_{\ell},g_r\ra=\la x,g_{2k-1},g_{2k-2}\ra\subseteq \la x,g_1,\dots,g_m\ra.$$

If one multiplies $x$ by the inverses of $(x_0)_{[v_{4k-3}]}$ for $k=2,\dots,\frac{m+1}{2}$, one remains with $(x_0)_{[v_{1}]}$. Thus, $(x_0)_{[v_{1}]}\in \la x,g_1,\dots,g_m\ra$ as well.

Next, one should consider the element $$y=\prod_{i=1}^{\mathclap{\frac{m+1}{2}}} (x_0)_{[v_{4i-1}]}.$$

If one applies the same arguments as above, one gets that for all $k\in\{1,\dots,\frac{m+1}{2}\}$, the element $(x_0)_{[v_{4k-1}]}$ is in $\la y,g_1,\dots,g_m\ra$. Combining, we get that for all $j\in\{1,\dots,m+1\}$, the element $(x_0)_{[v_{2j-1}]}\in H$  as necessary.\\

The proof for $m$ even is very similar. Recall that in that case, we have
$$x=\left[\prod_{i=1}^{\frac{m}{2}} (x_0)_{[v_{4i-3}]}\right](x_0)_{[v_{2m-1}]}\ \  \hbox{  and   }\ \   y=\left[\prod_{i=1}^{{\frac{m}{2}}} (x_0)_{[v_{4i-1}]}\right](x_0)_{[v_{2m+1}]}.$$
We apply Lemma \ref{x0} with $u\equiv v_{2m-1}$, $$h_{\ell}=\prod_{i=1}^{\frac{m}{2}} (x_0)_{[v_{4i-3}]}, \ \
h_r=\left[\prod_{i=1}^{{\frac{m-2}{2}}} (x_0)_{[v_{4i-1}]}\right](x_0)_{[v_{2m+1}]},$$ $g_{\ell}=g_m$ and $g_r=g_{m-1}$. One can check that all the conditions in Lemma \ref{x0} are satisfied. Notice, in addition, that $f_{\ell}$ from the lemma is equal to $x$ and $f_r=y$. Thus, Lemma \ref{x0} implies that %the subgroup generated by $f_{\ell}, f_{r}$, $g_{\ell}$ and ${g_r}$ contains $(x_0)_{[v_{2m-1}]}$.
$(x_0)_{[v_{2m-1}]}\in \la f_{\ell},f_r,g_{\ell},g_r\ra\subseteq H.$
Multiplying $x$ and $y$ by $((x_0)_{[v_{2m-1}]})^{-1}$, results in elements
$$x'=\prod_{i=1}^{\frac{m}{2}} (x_0)_{[v_{4i-3}]}\ \  \hbox{  and  }\ \    y'=\left[\prod_{i=1}^{{\frac{m-2}{2}}} (x_0)_{[v_{4i-1}]}\right](x_0)_{[v_{2m+1}]}.$$
Proceeding as in the case where $m$ is odd,  for each $k=2,\dots,\frac{m}{2}$
one can apply Lemma \ref{x0} with $u\equiv v_{4k-3}$, $h_{\ell}=h_{r}=x'((x_0)_{[4k-3]})^{-1}$, $g_{\ell}=g_{2k-1}$ and $g_r=g_{2k-2}$. As a result, one gets that the elements $ (x_0)_{[v_{4k-3}]}\in \la x',g_1,\dots,g_m\ra $ for $k=2,\dots,\frac{m}{2}$. Multiplying $x'$ by the inverses of $(x_0)_{[v_{4k-3}]}$ for $k=2,\dots,\frac{m}{2}$ shows that $({x_0})_{[v_1]}\in \la x',g_1,\dots,g_m\ra$.
 %$x'\in H$, this is also true for $k=1$.
In a similar way, one shows that $(x_0)_{[v_{4k-1}]}\in \la y',g_1,\dots,g_m\ra$ for $k=1,\dots,\frac{m-2}{2}$,
and as such so does $(x_0)_{[v_{2m+1}]}$. All together, we have that for all $j\in\{1,\dots,m+1\}$, the element $(x_0)_{[v_{2j-1}]}\in H$, which completes the proof of the lemma.
\end{proof}
The proof of Theorem \ref{thm:rational} is complete.
\end{proof}

\begin{Remark} \label{r:u} Suppose now that $U_1$ is not empty (but $U_3=\emptyset$). If $|U_1|=k$, then the numbers from $U_1$ separate $U_2$ into a disjoint union $U_{2,1}\cup ...\cup U_{2,k+1}$ of subsets (some of which might be empty). By the results of Section \ref{sec:alg}, $H_U$ is isomorphic to the direct product of subgroups $H_{U_{2,i}}$. Note that if $U_{2,i}=\emptyset$ then $H_{U_{2,i}}\cong F$.
%this gives a decomposition of $H_U$ as a direct product of subgroups $H_{U_{2,i}}$. Note that if $U_{2,i}=\emptyset$ then $H_{U_{2,i}}\cong F$.
\end{Remark}

 This allows us to compute the presentation of $H_U$ and, in particular, the minimal number of generators of that subgroup.

\begin{Theorem}\label{t:rat}
Let $U$ be a finite set of rational numbers in $(0,1)$, $U=U_1\cup U_2$ be its natural partition.  Then the smallest number of generators of $H_U$ is $2|U_1|+|U_2|+2$.
\end{Theorem}

\begin{proof} The fact that the smallest number of generators of $H_U$ cannot be smaller than $2|U_1|+|U_2|+2$ follows from Theorem \ref{thm:semi}. Let us prove the upper bound.
The proof is by induction on the number $n=|U_1|$. If  $n=0$, the result is Theorem \ref{thm:rational}. Assume that $n>0$ and $m=|U_2|$. Let $\alpha$ be the smallest number in $U_1$ and assume that there are $c\ge 0$ numbers from $U_2$ smaller than $\alpha$. From Remark \ref{r:u} it follows that $H_U$  is isomorphic to the direct product of $H_V$ and $H_W$, where $V=V_2$, $|V|=c$, $|W_1|=n-1, |W_2|=m-c$. Thus from the induction hypothesis we have that $H_V$ is generated by $c+2$ elements and $H_W$ is generated by $2(n-1)+m-c+2$ elements. Hence the direct product $H_U\cong H_V\times H_W$ is generated by $2n+m+2$ elements.
\end{proof}

\section{Finitely generated subgroups $H_U$ are undistorted} Recall that if $G$ is a group generated by a finite set $S$ and $H$ is a subgroup of $G$ generated by a finite set $T$, then the distortion function $\delta_{S,T}$ is the smallest function $\N\to \N$ such that if an element $h\in H$ is a product of $n$ elements of $S$, then it is a product of at most $f(n)$ elements of $T$. For fixed $G, H$ but different (finite) $S, T$, the functions $\delta_{S,T}$ are equivalent\footnote{Two functions $f,g\colon \N\to \N$ are called \emph{equivalent} if for some $c>1$, $\frac 1c f(\frac nc)-c \le g(n)\le cf(cn)+c$ for every $n\in \N$.}.  The subgroup $H$ is called \emph{undistorted} in $G$ if the distortion function is linear. %If $M$ is a finitely generated subgroup of $H$ such that $M$ is undistorted in $H$ and $H$ is undistorted in $G$, then $M$ is undistorted in $G$. Also, a direct product of subgroups of $G$ is undistorted if and only if each factor is undistorted
Although many subgroups of the Thompson group $F$ are known to be undistorted (see, for example, \cite{GuSa97,Burillo, GuSa99, WC}), $F$ has distorted subgroups \cite{GuSa99, DO}.

\begin{Theorem} \label{t:dist} Let $U$ be a finite set of rational numbers in $(0,1)$. Then the subgroup $H_U$ is undistorted in $F$.
\end{Theorem}

\proof
Let $U=\{\alpha_1,\dots,\alpha_n\}$. If $|U_1|=k$, then $U_1$ separates  $U_2$ into a union of $k+1$ disjoint subsets
$U_{2,1},\dots,U_{2,k+1}$. Let $i_1,\dots,i_{k}$ be the indexes such that $\alpha_{i_j}\in U_1$. Let $\alpha_0=0$, $\alpha_{n+1}=1$, $i_0=0$ and $i_{k+1}=n+1$.
By results of Section \ref{sec:alg},
$$H_U=B_{\{\alpha_{i_0},\dots,\alpha_{i_1}\}}\times
B_{\{\alpha_{i_1},\dots,\alpha_{i_2}\}}\times\dots\times
B_{\{\alpha_{i_{k-1}},\dots,\alpha_{i_k}\}}\times
B_{\{\alpha_{i_k},\dots,\alpha_{i_{k+1}}\}},
$$
where $B_{\{\alpha_{i_{j-1}},\dots,\alpha_{i_j}\}}=F_{[\alpha_{i_{j-1}},\alpha_{i_j}]}\cap H_{U_{2,j}}$ for $j=1,\dots,k+1$.

Since a direct product is undistorted if and only if each factor is undistorted, it suffices to prove that if $a<b$ belong to $(0,1)\cap\zz$ and $U'=U'_2$ is a set of rational numbers in $(a,b)$ then $F_{[a,b]}\cap H_{U'}$ is undistorted in $F$. We claim that $F_{[a,b]}$ is undistorted in $F$ and that $F_{[a,b]}\cap H_{U'}$ is undistorted in $F_{[a,b]}$. That will imply that $F_{[a,b]}\cap H_{U'}$ is undistorted in $F$.

Proposition 9 in \cite{Burillo} implies that $F_{[0,\frac{1}{2}]}$ and $F_{[\frac{1}{2},1]}$ are undistorted in $F$.
Clearly, one can replace $\frac{1}{2}$ by any number $\alpha\in(0,1)\cap\zz$ by conjugting $F$ by a function $f\in F$ such that $f(\frac{1}{2})=\alpha$. Similarly, let $g\in\pl_2(\mathbb{R})$ be a function such that $g(0)=0$, $g(\frac{1}{2})=a$ and $g(1)=b$, then
$F^{g}=F_{[0,b]}$ and ${F_{[\frac{1}{2},1]}}^g=F_{[a,b]}$. It follows that $F_{[a,b]}$ is undistorted in $F_{[0,b]}$. Since $F_{[0,b]}$ is undistorted in $F$ we get that $F_{[a,b]}$ is undistorted in $F$.

To prove that $F_{[a,b]}\cap H_{U'}$ is undistorted in $F_{[a,b]}$ we note that there is an isomorphism from $F_{[a,b]}$ to $F$ which maps $F_{[a,b]}\cap H_{U'}$ onto a subgroup $H_{U''}\le F$ where $U''=U''_2$ and $|U''|=|U'|$.
Clearly, it suffices to prove that $H_{U''}$ is undistorted in $F$. In other words, it suffices to prove the theorem for subsets $U$ such that $U=U_2$.

Let $U=U_2$. We shall use the proof of Theorem \ref{thm:rational}. Let $h\in H_U$ be a product of at most $n$ elements from $\{x_0,x_1\}$. Then, viewed as a reduced tree-diagram, $h$ has at most $cn$ vertices for some constant $c$. By the proof of Lemma \ref{KU}, $H_U$ is equal to $GK_UG$ where  $K_U$ is a direct product of $m+1$ copies $F_{[v_{2i-1}]}$ of $F$, $i=1,...,m+1$ (where $m=|U|$) and $G$ is  a subgroup generated by the elements $g_1,\dots,g_m$ defined before Lemma \ref{KU}.

Let us take the generators $[x_0]_{v_{2i-1}}, [x_1]_{v_{2i-1}}$ in each $F_{[v_{2i-1}]}$ and all the elements $g_1,\dots,g_m$ as the elements of a generating set $T$ of $H_U$.

The proof of Theorem \ref{thm:rational} gives us that by multiplying $h$ from the left and from the right by powers of $g_1,...,g_m$ one can get an element $h'$ in $K_U$.
The number of elements $g_1,...,g_m$ required for the product is bounded from above by the number of vertices of $h$.
Hence the element $h'$ has at most $cn+cn(c_1+....+c_m)$ vertices, where $c_1,...,c_m$ are constants depending on $g_1,...,g_m$. Then $h'=h_1\cdots h_{m+1}$ where $h_i\in F_{[v_{2i-1}]}$, hence all $h_i$ have pairwise disjoint supports. Then each  $h_i$ is represented by a diagram with at most $d_in$ vertices where $d_i$ is a constant.  By Property B of Burillo (see \cite{Burillo, AGS}),
 $h_i$ is a product of at most $d_i'n$ generators from $\{[x_0]_{v_{2i-1}},[x_1]_{v_{2i-1}}\}$. Hence $h$  is a product of at most $c'n$ generators of $H_U$ for some constant $c'$.
\endproof

\section{Isomorphism vs conjugacy}\label{s:conj}

In this section we show that the isomorphism between $H_U$ and $H_V$ (provided $\tau(U)\equiv\tau(V)$) is induced by conjugacy in some bigger group. In fact we construct a chain  $F< \F < \bar F<\Homeo([0,1])$ such that  $\F$ is similar to $F$ and consists of possibly infinite tree-diagrams, $\bar F$ is the completion of $F$ with respect to  a certain natural metric, and $H_U, H_V$ are conjugate inside $\F$. This strengthens Theorem \ref{th:m1}.

\subsection{The completion  of $F$ with respect to the Hamming metric}

Let $\mu$ be the standard Lebesgue  measure on $[0,1]$. Consider the following metric on the group $F$:
$$\dist_H(f,g)=\mu(\Supp(fg\iv))+\mu(\Supp(f\iv g))$$
or, equivalently,
$$\dist_H(f,g)=\mu(\{x\mid f(x)\ne g(x)\}) + \mu(\{x\mid f^{-1}(x)\ne g^{-1}(x)\}).$$

Clearly, $\dist_H$ is a distance function on $F$.

The metric $\dist_H$ is similar to the standard Hamming metric on the symmetric group $S_n$ although, unlike the Hamming metric on $S_n$, $\dist_H$ is not invariant with respect to left or right multiplication by elements of $F$. Thus we shall call $\dist_H$ the \emph{Hamming} metric on $F$.

\begin{Remark}\label{r:new}
It is easy to show that the group operations of $F$ (the multiplication and the inverse) are continuous with respect to $\dist_H$. This follows from the fact that for every $f\in F$ and every $\epsilon>0$ there exists $\delta>0$ such that if $\mu(S) < \delta$, then $\mu(f(S))<\epsilon$ (one can take $\delta=\frac{\epsilon}{2^n}$ where $2^n$ is the maximal slope of $f$). Note that this fact is not true for arbitrary $f\in \Homeo([0,1])$. Hence, although $\dist_H$ can be obviously extended to the whole $\Homeo([0,1])$, the multiplication in $\Homeo([0,1])$ is not continuous with respect to $\dist_H$.
\end{Remark}

The Hamming metric is of course related in a standard way to the norm $|f|=\mu(\Supp(f))$.
A similar norm on the group of diffeomorphisms of arbitrary manifolds has been considered in \cite[Example 1.19]{BIP}. For the group of measure preserving maps of a measure space, this is sometimes called the \emph{uniform metric} \cite{CPestov}.

\begin{Definition}
We denote by $\bar F$ the (Ra\v ikov) completion of $F$ with respect to $\dist_H$ \cite[Section 3.6]{AT}. It consists of Cauchy sequences $(f_n), n\ge 1,$ of elements of $F$ with two sequences $(f_n), (g_n)$ being equivalent if
$\lim_{n\rightarrow\infty}(\dist_H(f_n,g_n))=0$.
\end{Definition}

\begin{Theorem} \label{t:at} The standard embedding $F\to \Homeo([0,1])$ extends to an embedding $\bar F\to \Homeo([0,1])$.
\end{Theorem}

\proof Let a sequence of functions $(g_m)$ from $F$ be Cauchy with respect to the Hamming metric.
We claim that then the pointwise limit of the sequence $(g_m)$ exists.
%One can get a ``limit'' monotone function on $[0,1]$ as follows.
Namely, for every number $x \in [0,1]$, consider the sequence of real numbers $(g_m(x))$, $m\ge 1$. We claim that $(g_m(x))$ is a Cauchy sequence of numbers from $[0,1]$ (with respect to the standard metric on $[0,1]$) and hence has a limit in $[0,1]$.
Indeed, for each $\epsilon$ there is some $n\in \N$ such that for all $m_1,m_2>n$ we have $\dist_H(g_{m_1},g_{m_2})<\epsilon$.
It suffices to prove that for the same $n$, the diameter of the set $B_n=\{g_m(x):m>n\}$ is at most $\epsilon$.
Assume that the diameter of $B_n=\{g_m(x):m>n\}$ is greater than $\epsilon$ and let $a,b$ be elements of $B_n$ such that $a<b$ and $b-a>\epsilon$. Let $m_1,m_2>n$ be such that
$g_{m_1}(x)=a$ and $g_{m_2}(x)=b$.
Clearly, $g_{m_1}^{-1}([a,b])$ is contained in $[x,1]$.
Similarly, $g_{m_2}^{-1}([a,b])$ is contained in $[0,x]$.
As such, for all $y$ in the open interval $(a,b)$ we have $g_{m_1}^{-1}(y)\neq g_{m_2}^{-1}(y)$.
Therefore, $\dist_H(g_{m_1},g_{m_2})\ge b-a>\epsilon$, in contradiction to the assumption.

Thus, $(g_m(x))$  has a limit in $[0,1]$. Let $g(x)$ be the limit of $(g_m(x))$ for each $x\in [0,1]$.
Then $g(0)=0, g(1)=1$. Moreover $g(x)$ is a non-decreasing  function. Indeed, if $x<y$, then for each $m$ we have $g_m(x)<g_m(y)$ and in particular the limits satisfy $g(x)\le g(y)$.

Note that the sequence $(g_m\iv)$, $m\ge 1$, is also a Cauchy sequence (by the definition of $\dist_H$). Hence we can define a function $g'(x)$ as $\lim_{m\to \infty} g_m\iv(x)$.
We claim that $gg'(x)=x$ for all $x\in [0,1]$. Indeed, assume that for some $x$, $gg'(x)=y\neq x$. We can assume that $y>x$, the argument for $y<x$ being similar.
Let $\epsilon=\frac{y-x}{4}$ and let $n\in\N$ be such that for  all $m_1,m_2> n$, $\dist_H(g_{m_1},g_{m_2})<\epsilon$. Since $g_m^{-1}(g(x))$ has limit $g'(g(x))=y$, there is
%$n_2\in\N$ such that for all $m>n_2$ we have
$m_1>n$ such that $|g_{m_1}^{-1}(g(x))-y|<\epsilon$. %Fix $m_1>\max\{n_1,n_2\}$ and
Let $y_1=g_{m_1}^{-1}(g(x))$. By assumption, $|y_1-y|<\epsilon$. Therefore, for $z=\frac{x+y}{2}$, we have $z\in (x,y_1)$.
Let $c=g_{m_1}(z)<g_{m_1}(y_1)=g(x)$.  Clearly, $g_{m_1}([x,z])\subseteq [0,c]$.
Since the sequence $g_m(x)$ converges to $g(x)$ and $c<g(x)$, for some $m_2>m_1$ we have $g_{m_2}(x)>c$.
In particular, $g_{m_2}([x,z])\subseteq [c,1]$.
It follows that $[x,z]\subseteq \{t:g_{m_1}(t)\neq g_{m_2}(t)\}$. Thus, $\dist_H(g_{m_1},g_{m_2})\ge z-x=2\epsilon$, in contradiction to $m_1,m_2$ being greater than $n$.
Thus, $gg'(x)=x$ for all $x\in [0,1]$. A similar argument shows that $g'g(x)=x$. Thus, $g$ is a one-to-one increasing function $[0,1]\to [0,1]$, hence $g$ is an element from $\Homeo([0,1])$.

Now let $(h_m)$ be a Cauchy sequence in $F$ equivalent to $(g_m)$ and let $h$ be the pointwise limit of $(h_m)$, that is, for every $x\in [0,1]$, $h(x)=\lim_{m\rightarrow \infty}h_m(x)$. We claim that $g=h$. Indeed, assume by contradiction that for some $x\in (0,1)$, $g(x)\neq h(x)$. Without loss of generality we can assume that $g(x)<h(x)$.
Let $a,b$ be such that $g(x)<a<b<h(x)$. By definition of $g(x)$ and $h(x)$, for every large enough $m$, we have $g_m(x)<a$ and $h_m(x)>b$. Thus, $g_m^{-1}([a,b])\subseteq[x,1]$ and $h_m^{-1}([a,b])\subseteq [0,x]$. As before, that implies that for all large enough $m$, $\dist_H(g_m,h_m)\ge b-a$, in contradiction to $(g_m)$ and $(h_m)$ being equivalent Cauchy sequences.

Thus, mapping a Cauchy sequence $(g_m)$ in $F$ to its pointwise-limit as defined above gives a well defined mapping $\psi$ from $\bar{F}$ to $\Homeo([0,1])$. %It is easy to see $\psi$ is a homomorphism.

Note that the map $\psi\colon \bar F\to \Homeo([0,1])$ is clearly a homomorphism whose restriction to $F$  coincides with the standard embedding of $F$ into $\Homeo([0,1])$.

We claim that $\psi$ is injective.
To prove the claim, assume that $(f_m)$ is a Cauchy sequence of elements in $F$ with $\lim_{m\to \infty} f_m(x)=x$ for every $x\in [0,1]$. We need to prove that then $$\lim_{m\to \infty} \dist_H(f_m,\1)=0.$$
Suppose that this equality is not true. Then we can assume that for infinitely many $m, \mu(\Supp(f_m))>d$ for some $d > 0$. Restricting to a subsequence, we can assume that for all $m, \mu(\Supp(f_m))>d$.
%Let $\epsilon<d$.
Since $(f_m)$ is Cauchy, there is $n\in\N$ such that for all $m>n$,
$\dist_H(f_{m},f_{n})<d$.

For each $k\in\N$, let
$$C_k=\bigcap_{m>k}\Supp(f_nf_m^{-1}).$$
Then the sequence $C_k$, $k\in\N$ is an increasing sequence of closed subsets of $[0,1]$.
We claim that $\bigcup_{k\in \N} C_k$ contains the interior of $\Supp(f_n)$, denoted by $\mathrm{Int}(\Supp(f_n))$.
Indeed, let $x\in \mathrm{Int}(\Supp(f_n))$ and assume by contradiction that for all $k$, $x\notin C_k$. Then for all $k\in\N$, there is some $m>k$ such that $x\notin \Supp(f_nf_m^{-1})$, so that $f_m(x)=f_n(x)$. It follows, that the value $f_n(x)$ appears infinitely many times in the sequence $(f_m(x))$. Since the sequence $(f_m(x))$ is convergent, we must have $\lim_{m\rightarrow \infty}f_m(x)=f_n(x)\neq x$, in contradiction to the assumption.

Thus, the union of $C_k$, $k\in\N$ contains the interior of the support of $f_n$. Since $C_k$ is increasing, we have
 $$\lim_{k\rightarrow\infty}\mu(C_k)=\mu(\bigcup_{k\in\N} C_k)\ge\mu(\Supp(f_n))>d.$$
It follows, that there is some $k>n$ such that $\mu(C_k)>d$. This clearly implies that $\dist_H(f_k,f_n)>d$ in contradiction to the choice of $n$.
%Let $S_m$ be the support of $f_m$. Then $S_m$ is a finite union of closed intervals.
\endproof

\subsection{A subgroup of $\bar F$}

Let $T$ be an infinite binary tree. Then there is a natural left-to-right (lexicographic) order on the  branches of $T$, and a natural subdivision of the unit interval into possibly infinite number of intervals corresponding to the  branches of the tree. Infinite branches of $T$  correspond to intervals with empty interior of that subdivision. Other intervals have finite dyadic endpoints.

\begin{Definition}\label{d:9} Consider the set $\F$ of triples $(T_+, T_-,\phi)$ where $T_+, T_-$ are binary trees and $\phi$ is a bijection from the set of branches of $T_+$ to the set of branches of $T_-$ satisfying the following properties:

\begin{enumerate}
\item $T_+$ and $T_-$ have the same finite number of infinite branches;
\item If a branch $p$ in $T_+$ is to the left of branch $q$, then $\phi(p)$ is to the left of $\phi(q)$ in $T_-$;
\item $\phi$ takes infinite branches to infinite branches.
\end{enumerate}

If $T_+$ and $T_-$ are finite trees, then the function $\phi$ taking leaves of $T_+$ to leaves of $T_-$ is uniquely defined. So $F$ is naturally a subset of $\F$.

One can extend the equivalence relation (inserting and reducing pairs of common carets as in Section \ref{sec:red}) and the group operations of $F$, multiplication and  taking inverse (see Remark \ref{r:000}), to the set $\F$  making $\F$ a group containing $F$ as a subgroup.
\end{Definition}

Every triple $(T_+,T_-,\phi)$ from the group $\F$ corresponds to a homeomorphism  $\gamma=\gamma(T_+,T_-,\phi)$ of $[0,1]$ that takes linearly each interval $[p]$ of the partition corresponding to a branch $p$ of $T_+$ to the interval $[\phi(p)]$ corresponding to the branch $\phi(p)$ of $T_-$. Properties 1, 2, 3 of Definition \ref{d:9} imply that $\gamma$ is indeed a homeomorphism of $[0,1]$. Note that for every point $x$ of $[0,1]$, except for the finite number of points corresponding to infinite branches of $T_+$, the homeomorphism $\gamma$ has left and right derivatives at $x$ which are powers of $2$ and all break points of the derivative of $\gamma$, except for the finite set of points, are from $\zz$. As for the R. Thompson group $F$, this gives an embedding of $\F$ into $\Homeo([0,1])$ which extends the natural embedding of $F$. We shall identify $\F$ with its image in $\Homeo([0,1])$.

\begin{Theorem}\label{t:m4}
The image of $\F$ in $\Homeo([0,1])$ is inside $\bar F$, viewed as a subgroup of $\Homeo([0,1])$.
% The identity isomorphism $F\to F$ extends to an injective homomorphism $\F\to \bar F$.
\end{Theorem}

\proof

Let $g=(T_+,T_-,\phi)$ be a triple in $\F$. It suffices to show that there is a sequence $(g_m), m\ge 1,$ of elements in $F$ which converges to $g$ in the Hamming metric on $\Homeo([0,1])$.
%Every triple $g=(T_+,T_-,\phi)$ in $\F$ can be viewed naturally as the ``limit'' of tree-diagrams in $F$. Indeed,
Let $\omega_i\rightarrow \omega'_i$, $i=1,\dots,n$ be the pairs of infinite branches of $(T_+,T_-,\phi)$.
For every $m\in\mathbb{N}$ and every $i=1,...,n$ let $u_{i,m}$ (resp. $v_{i,m}$) be the prefix of $\omega_i$ (resp. $\omega'_i$) of length $m$. For each $i$, the sequence $[u_{i,m}]$, $m\in \N$ (resp, $[v_{i,m}]$, $m\in\N$) is a nested sequence of intervals with intersection $\{.\omega_i\}$ (resp. $\{.\omega_i'\}$).
Let $m\in\mathbb{N}$. Since $g$ is continuous and $g(.\omega_i)=.\omega_i'$, $i=1,\dots,n$, there exists some $m'\ge m$ such that for all $i=1,\dots,n$, $g([u_{i,m'}])\subseteq [v_{i,m}]$. Consider the subtree $S_{m'}$ of $T_+$ of all vertices $x$ such that none of the $u_{i,m'}$'s is a strict prefix of the path from the root to $x$. Then $S_{m'}$ is a rooted binary subtree of $T_+$. Clearly, $S_{m'}$ does not have any infinite branch. Since $S_{m'}$ is a  binary tree, that implies that it has only finitely many branches. Clearly, $n$ of these branches are $u_{1,m'},\dots,u_{n,m'}$. %Let $w_1,\dots,w_m$ be the rest of the branches of $S$. Then $w_1,\dots,w_m$ are separated by the branches $u_1,\dots,u_n$ into sets of branches $w_{j,1},\dots,w_{j,m_j}$ for $j=1,\dots,n+1$ and $m_1+\dots m_{n+1}=m$.
%For each branch $w
The rest of the branches of $S_{m'}$ are some branches of $T_+$. For any branch $u$ of $S_{m'}$ which is not one of the branches $u_{1,m'},\dots, u_{n,m'}$, there is a branch $v$ of $T_-$ such that $v=\phi(u)$.
There exists a function $g_m\in F$ whose tree-diagram contains all these branches $u\to \phi(u)$. In particular, $g_m$ coincides with $g$ on $[0,1]\setminus[u_{1,m'}]\cup\dots\cup[u_{n,m'}]\supseteq[0,1]\setminus[u_{1,m}]\cup\dots\cup[u_{n,m}] $. The choice of $m'$ implies that $g_m^{-1}$
coincides with $g^{-1}$ on $[0,1]\setminus[v_{1,m}]\cup\dots\cup[v_{n,m}]$.
It follows that the sequence $(g_m), m=1,2,...,$ converges to $g$ in the Hamming metric on $\Homeo([0,1])$.
%Since $(g_m), m\ge 1$, is a Cauchy sequence of elements of $F$ with respect to $\dist_H$, it can be viewed as an element $\bar g$ from $\bar F$. Note that if we choose different elements $g'_m$ instead of $g_m$ then $\lim_{m\to \infty} \dist_H(g_m,g_m')=0$ and so the sequence $(g_m'), m\ge 1$ determines the same element $\bar g\in \bar F$.
%
%This gives us a well defined function $\gamma\colon \F\to \bar F$ which maps each $g$ to $\bar g$. Note that if $g\in F$, that is if $g$ has no infinite branches, then the sequence $(g_m),m\ge 1$ is a constant sequence, and so $\gamma(g)=g$. The map $\gamma$ is injective as Cauchy sequences converging to different homeomorphisms $g\in\F\subseteq \mathrm{Homeo}([0,1])$ cannot be equivalent (i.e. cannot give the same element in $\bar F$).
%
%To prove that $\gamma$ is a homomorphism, take two elements $f,f'\in \F$ and the corresponding sequences $(f_m), (f_m'), m\ge 1,$ constructed as above and converging to $f,f'$ with respect to $\dist_H$. Let $(g_m), m\ge 1,$ be a sequence of elements of $F$ converging to $g=ff'$, constructed as above. Then $\lim_{m\to \infty} \dist_H(f_mf_m', g_m)=0$ and $(f_mf_m'), m\ge 1,$ is a Cauchy sequence.  Hence the sequences $\bar f\bar f'$ and $(g_m), m\ge 1$ are equal in $\bar F$. Thus $\gamma(f)\gamma(f')=\gamma(ff')$.
\endproof

%\begin{Remark} Note that $\bar F$ does  not naturally embed into $\Homeo([0,1])$. Indeed, consider a sequence of functions $(f_n), n\ge 1,$ from $F$ such that $f_n$ takes $t$ to $\frac t2$ provided $t\in [0,\frac12 t-\frac 1{2^{n+1}}]$ and to $\frac 12 t+\frac12$ provided $t\in [\frac 12+\frac 1{2^{n+1}},1]$. Then $(f_n)$ is a Cauchy sequence with respect to $\dist_H$ but the limit function is not continuous at $\frac 12$. Obviously $(f_n)$ is not equivalent to a Cauchy sequence converging to a continuous function.
%\end{Remark}

\begin{Theorem}\label{t:pl}
If $U$ and $V$ are two finite sets of numbers from $(0,1)$ and $\tau(U)\equiv \tau(V)$, then the subgroups $H_U$ and $H_V$ are conjugate in $\F$.
\end{Theorem}

To prove Theorem \ref{t:pl}, we need the following two lemmas.

\begin{Lemma}\label{111}
Let $\alpha=.pu^\mathbb{N}$ and $\beta=.pv^{\mathbb{N}}$ be two rational numbers in $(0,1)$ which are not in $\zz$.
Then the stabilizers $H_{\{\alpha\}}$ and $H_{\{\beta\}}$ are conjugate in $\F$. Moreover, the conjugator $f$ can be chosen to have support  in $[p]$.
\end{Lemma}

\begin{proof}
We can assume that $u$ and $v$ are minimal periods of $\alpha$ and $\beta$.

Let $(R_+,R_-)$ be a reduced tree-diagram with pairs of branches $u_i\rightarrow v_i$, $i\in\mathbb{N}$, such that for some $k\in\{2,\dots,n-1\}$, $u_k=u$ and $v_k=v$.
We let $f_1=(T_+^1,T_-^1)$ be the copy of $(R_+,R_-)$ in $F_{[p]}$. The pairs of branches of $(T_+^1,T_-^1)$ are of the form $pu_i\rightarrow pv_i$ for $i=1,\dots,n$, along with pairs of branches $b\rightarrow b$ for some words
$b$ such that $p$ is not a prefix of $b$.

We construct tree-diagrams $(T_+^j,T_-^j)$ inductively for $j\in\mathbb{N}$, so that for each $j$, the tree-diagram $(T_+^j,T_-^j)$ has a pair of branches $pu^j\rightarrow pv^j$.
For $j=1$, we are done. If $(T_+^j,T_-^j)$ is already constructed,  we attach a copy of the tree $R_+$ to the end of the branch $pu^{j}$ of $T_+^{j}$ and a copy of the tree $R_-$ to the end of the branch $pv^{j}$ of the tree $T_-^{j}$. We let $(T_+^{j+1},T_-^{j+1})$ be the resulting tree-diagram.

Let $f_j=(T_+^j,T_-^j)$. %Clearly, for all $j>1$, $f_j$ coincides with $f_{j-1}$ on $[0,1]\setminus[pu^{j-1}]$. % Notice also, that for all $j$, the function $f_j$ has the pairs of branches $pu^{j-1}u_i\rightarrow pv^{j-1}v_i$ for $i=1,\dots,n$.
We let $f$ be the ``limit'' of the tree-diagrams $f_j$ for $j\in\mathbb{N}$, in the following sense. %For every
%$x\neq .pu^{\mathbb{N}}$, there is $j_0$ such that for all $j\ge j_0$, all functions $f_j$ coincide on $x$. We let $f(x)=f_{j_0}(x)$. We let $f(.pu^{\mathbb{N}})=.pv^{\mathbb{N}}$.
%Note that for all $j\in\N$, $f|_{[0,1]\setminus[pu^j]}=f_j$.
The triple $f$ can be defined by listing its  pairs of branches $p\to q$ where $p$ and $q$ are either both finite or both infinite, and $\phi(p)=q$.

The pairs of branches of $(T_+,T_-,\phi)$ are $pu^ju_i\rightarrow pv^jv_i$ for $j\ge 0$, $i\in\{1,\dots,n\}\setminus\{k\}$;  pairs of branches $b\rightarrow b$ for some words $b$ such that $p$ is not a prefix of $b$ (namely, the pairs of branches of $(T_+^1,T_-^1)$ of this form) and the pair of infinite branches $pu^{\mathbb{N}}\rightarrow pv^{\mathbb{N}}$. It is obvious that $f\in \F$ and that, as an element of $\Homeo([0,1])$,  $f$ has support in $[p]$.

%Notice also that $f$ is the function represented by the infinite tree diagram $(T^+,T^-)$ which is the natural limit of the tree diagrams $(T_+^j,T_-^j)$. The pairs of branches of $(T_+,T_-)$ are $pu^ju_i\rightarrow pv^jv_i$ for $j\ge 0$, $i\in\{1,\dots,n\}\setminus\{k\}$; branches $b\rightarrow b$ for some words $b$ such that $p$ is not a prefix of $b$
%and the pair of infinite branches $pu^{\mathbb{N}}\rightarrow pv^{\mathbb{N}}$. It is obvious that $f\in \F$ and that $f$ has support in $[p]$.

We claim that ${H_{\{\alpha\}}}^f=H_{\{\beta\}}$. It is enough to  show the inclusion $\subseteq$ (to prove the other inclusion one would just replace $f$ by $f\iv$).  Let $g\in H_{\{\alpha\}}$. We view $f$ as a function from $\Homeo([0,1])$. By definition, $f(\alpha)=\beta$.
It is obvious that $g^f$ fixes $\beta$ since $g$ fixes $\alpha$ and $f(\alpha)=\beta$.

It remains to show that $g^f\in F$, that is, as a pair of trees, $g^f$ has only finitely many pairs of branches. To prove that, it is enough to find a partition of $[0,1]$ into a finite number of intervals $I_1,...,I_s$ with finite dyadic endpoints and  elements $g_1, ..., g_s$ from $F$ such that for every $k=1,...,s$, $f$ coincides with $g_k$ on $I_k$.

Since $g(\alpha)=\alpha$, there are natural numbers $m_1, m_2\ge 1$ such that $g$ has a pair of branches
$pu^{m_1}\rightarrow pu^{m_2}$ (see Lemma \ref{rational}). It is easy to see that by the definition of $f$, on
$[0,1]\setminus [pv^{m_1}]$ (which is a union of two intervals), the function $g^f$ coincides with $f_{m_1}^{-1}gf_{m_2}$. Consider a
number $x$ from $[pv^{m_1}]$. Then $x=.pv^{m_1}\omega$ for some infinite binary word $\omega$.
%
%
%Then it suffices to show that $g^f$ has a pair of branches $pv^{m_1}\rightarrow pv^{m_2}$. Indeed, consider a number $x=.pv^{m_1}\omega$ for some infinite binary word $\omega$.
It is easy to  check that $g^f$ maps $x$ to $y=.pv^{m_2}\omega$. Thus if $g'$ is any element from $F$ which has the pair of branches $pv^{m_1} \to pv^{m_2}$, then  $g^f$ coincides with $g'$ on $[pv^{m_1}]$. Therefore $g^f\in F$.
\end{proof}

\begin{Lemma}\label{222}
Let $\alpha$ and $\beta$ be two irrational numbers in $(0,1)$. Then $H_{\{\alpha\}}$ and $H_{\{\beta\}}$ are conjugate in $\F$. If $\alpha$ and $\beta$ belong to a dyadic interval $[p]$, then the conjugator of $H_{\{\alpha\}}$ and $H_{\{\beta\}}$ in $\F$ can be taken to have support in $[p]$.
\end{Lemma}

\begin{proof}
The proof is similar to the proof of Lemma \ref{der}. Let $\alpha=.\omega$ and $\beta=.\omega'$. For $i\in\mathbb{N}$, let $u_i$ (resp. $v_i$) be the prefix of $\omega$ (resp. $\omega'$) of length $i$. We note that $[u_i]$ (resp. $[v_i]$) is a nested sequence of dyadic intervals whose intersection is $\{\alpha \}$ (resp. $\{\beta\}$).
%We let $[u_i]$, $i\in\mathbb{N}$ be a sequence of nested dyadic intervals such that $\bigcap_{i\in\mathbb{N}} [u_i]=\{\alpha\}$. Similarly, let $[v_i]$, $i\in\mathbb{N}$ be a sequence of nested intervals such that $\bigcap_{\i\in\mathbb{N}}[u_i]=\{\beta\}$.
%We note that $H_{\{\alpha\}=\bigcap_{i\in\mathhb{N}}H_{[u_i]}$

We construct a sequence of elements $f_i=(T_+^i,T_-^i)$ in $F$ such that for any $i$, $(T_+^i,T_-^i)$ has the pair of branches $u_i\rightarrow v_i$. For $i=1$ we let $(T_+^1,T_-^1)$ be a tree-diagram with a pair of branches $u_1\rightarrow v_1$. For any $i>1$, we let $(T_+^i,T_-^i)$ be a tree-diagram which has a pair of branches $u_i\rightarrow v_i$ and has all the pairs of branches of $(T_+^{i-1},T_-^{i-1})$, other than the pair of branches $u_{i-1}\rightarrow v_{i-1}$. This is possible by Remark \ref{r:choice}.

We let the triple $f=(T_+,T_-,\phi)$ be defined (as a collection of pairs of branches) as follows.
%the limit of the tree-diagrams $f_i=(T_+^i,T_-^i)$ in the following sense.
The only pair of infinite branches in $f$ is  $\omega\rightarrow \omega'$. The pairs of finite branches are, all the pairs of branches $u\rightarrow v$ of $(T_+^i, T_-^i)$ other than the pair $u_i\rightarrow v_i$, for  each $i\in\mathbb{N}$.  Note that every finite pair of branches of $f$ occurs in all but finitely many tree-diagrams $(T_+^i,T_-^i)$. Clearly, for $i\in\mathbb{N}$, $f$ coincides with $f_i$ on $[0,1]\setminus[u_i]$.

We claim that $H_{\{\alpha\}}^{f}=H_{\{\beta\}}$. It is enough to  show the inclusion $\subseteq$.  Let $g\in H_{\{\alpha\}}$. By definition, $f(\alpha)=\beta$. Thus, $g^f$ fixes $\beta$ since $g$ fixes $\alpha$. To show that $g^f\in F$, we note that by Corollary \ref{cor:irr}, since $g$ fixes $\alpha$, it fixes a neighborhood $[u_i]$ of $\alpha$ for a large enough $i\in\mathbb{N}$. Since $f$ and $f_i$ coincide on $[0,1]\setminus[u_i]$, in particular they coincide on the support of $g$. It follows that $g^f=g^{f_i}\in F$.

If $\alpha$ and $\beta$ belong to $[p]$, then for some $j\in\mathbb{N}$ we have $u_j=v_j=p$ and one can construct the functions $f_1,\dots,f_j$ in the sequence above so that they act as the identity on $[0,1]\setminus[p]$.
Continuing the construction of the sequence $\{f_i\}$ as above, one gets that the limit function $f$ has support in $[p]$.
\end{proof}

\begin{proof}[Proof of Theorem \ref{t:pl}]
Let $U=U_1\cup U_2\cup U_3$ be the natural partition of $U$.
For each $\beta\in U_2\cup U_3$ we choose a small  dyadic interval $[p_{\beta}]$ such that $\beta\in [p_\beta]$ (i.e., $\beta=.p_\beta\omega$ for some infinite binary word $\omega$). We can assume that the intervals $[p_\beta]$, $\beta\in U_2\cup U_3$ are pairwise disjoint and that each of these intervals contains exactly one element of $U\cup \{0,1\}$, the  number $\beta$.

Since $\tau(U)\equiv\tau(V)$, using conjugation in $F$ if necessary (as in the proof of Theorem \ref{th:m1}), we can assume that the set $V=V_1\cup V_2\cup V_3$ satisfies the following conditions.
\begin{enumerate}
\item[(1)] $V_1=U_1$.
\item[(2)] For each $\delta\in V_2\cup V_3$, $\delta\in [p_\beta]$, where $\beta$ occupies the same position in the ordered set $U$ as $\delta$ does in $V$.
\end{enumerate}

By Lemmas \ref{111} and \ref{222}, for each $\beta\in U_2\cup U_3$ there is a homeomorphism $f_\beta\in \F$ with support in $[p_\beta]$ such that ${H_{\{\beta\}}}^{f_\beta}=H_{\{\delta\}}$. %where $\beta$ and $\delta$ are corresponding numbers in $U$ and $V$.
It is easy to check that $f=\prod_{\beta\in U_2\cup U_3} f_{\beta}\in \F$ conjuagtes $H_U$ to $H_V$.
\end{proof}

\subsection{$\F$ and $\bar F$ are not amenable}

\begin{Theorem}\label{t:freesub}
The group $\F$ contains a non-Abelian free subgroup.
\end{Theorem}

\proof
Let $O$ be the set of all finite tuples of elements of $F$. Since $O$ is countable, there is a bijection $\phi\colon\mathbb{N}\to O$.
One can list the elements of $F$ using the function $\phi$, by listing the elements of the tuple $\phi(1)$, followed by the elements of the tuple $\phi(2)$ and so on. Clearly, every element of $F$ is listed infinitely many times.
Also, if we associate a function $\psi\colon \mathbb{N}\to F$ to this listing, then for any tuple $t$ in $O$ of length $k$ there is some $n\in\mathbb{N}$ such that $t=(\psi(n),\dots,\psi(n+k-1))$.
For each $i\in\mathbb{N}$, we let $(R_+^i,R_-^i)$ be the reduced tree-diagram of $\psi(i)$.

Let $T$ be the minimal infinite binary tree with $0^\N$ as a branch
(i.e., the branches of $T$ are $0^\N$, $0^k1$ for all $k\ge 0$).
We construct an element of $\F$ as follows.
Let $T_+, T_-$ be two copies of $T$.
For each $k>0$ we attach to the tree $T_+$ the tree $R_+^k$ at the end of the branch $0^k1$.
Similarly, to the end of the branch $0^k1$ of $T_-$ we attach the tree $R_-^k$.
We denote the resulting tree-diagram with the natural mapping $\phi$ by $(T_+,T_-,\phi)$ and let $g$ be the function in $\F$ represented by it.
We note that for all $k\ge 0$, if $\alpha\in [0^k1]$ then $g(\alpha)\in [0^k1]$.
In addition, for any $k\in\mathbb{N}$ and any number $\alpha=.0^k1\omega$, we have
$g(\alpha)=g(.0^k1\omega)=.0^k1\omega'$, where $(\psi(k))(.\omega)=.\omega'$.
In other words, when restricted to the interval $[0^k1]$, $g$ coincides with the copy of $\psi(k)$ in $F_{[0^k1]}$.

Let $f$ be an element of Thompson group $F$ with a pair of branches $00\rightarrow 0$.
We claim that the group generated by $g$ and $f$ is free.
Assume by contradiction that $w(x,y)$ is a word in $\{x,y,x^{-1},y^{-1}\}$ such that $w(x,y)$ is not trivial in the free group over $\{x,y\}$ and such that $w(g,f)=1$.
Since for every $k>1$, $g$ maps the interval $[0^k1]$ onto itself and $f$ maps $[0^k1]$ onto $[0^{k-1}1]$, it is easy to see that the sum of powers of $y$ in the word $w(x,y)$ must be $0$.
Thus, $w(x,y)$ is equivalent in the free group to a word in conjugates of $x$ by powers of $y$.
Clearly, we can assume that all conjuagtors are positive powers of $y$.

We note that for all $k,j\in\mathbb{N}$, the function $g^{f^j}$ restricted to $[0^k1]$ coincides with the copy of $\psi(k+j)$ in $F_{[0^k1]}$.
Indeed, for any $\alpha=.0^k1\omega$,

$$\begin{array}{r}
g^{f^j}(\alpha)=f^j(g(f^{-j}(\alpha)))=f^j(g(f^{-j}(.0^k1\omega)))=f^j(g(.0^{k+j}1\omega))=f^j(.0^{k+j}1\omega')\\=.0^k1\omega',
\end{array}$$
where $\omega'$ here is such that $(\psi(k+j))(.\omega)=.\omega'$.

One can think of the situation as follows. If one fixes $k\in\mathbb{N}$ and considers only the dyadic interval $[0^k1]$ , then for every $j>0$, the function $g^{f^j}$ restricted to the chosen interval behaves exactly like (the relevant copy of) the function $\psi(k+j)$ of $F$.
As noted above, the word $w(f,g)$ can be viewed as a word $w'$ in conjugates of $g$ by positive powers of $f$.
If $\ell$ different conjugates $g^{f^{r_1}},\dots,g^{f^{r_{\ell}}}$ participate in the word, we can replace each one of them by a letter in an alphabet $\{z_1,\dots,z_\ell\}$ and consider the resulting word $w'(z_1,\dots,z_\ell)$. Since $F$ does not satisfy any law (see \cite[Theorem 5.6.37]{Sapirbook}) there is a sequence of elements $f_1,\dots,f_{\ell}$ in $F$ such that $w'(f_1,\dots,f_{\ell})\neq 1$.
It remains to notice that since the function $\psi$ enumerates all tuples of elements of $F$, one after the other, it is possible to choose $k$ such that $\psi(k+r_i)=f_i$ for all $i=1,\dots,\ell$.
Then the word $w'$ in the conjugates $g^{f^{r_1}},\dots,g^{f^{r_{\ell}}}$, behaves exactly as the (copy of the) function
$w'(f_1,\dots,f_{\ell})$ in the interval $[0^k1]$. As noted, it is not the identity.
\endproof

\section{Open problems}\label{open}

\subsection{The number of isomorphism classes of maximal subgroups of $F$}

 Note that in every non-cyclic countable free group every maximal subgroup of infinite index is free of countable rank. Thus, even though the set of maximal subgroups of infinite index of a non-cyclic free group is of cardinality continuum, there is only one isomorphism class of these subgroups.  From the results of \cite{Sav1}  it follows that the set of maximal subgroups of infinite index of $F$ also has cardinality continuum (it contains all subgroups $H_{\{\alpha\}}$, $\alpha\in (0,1)$). Still, the results of this paper and \cite{GolanSapir} show that,  up to isomorphism, only $4$ maximal subgroups of  $F$ of infinite index are known. The representatives of isomophism classes are $H_{\left\{\frac12\right\}}$, $H_{\left\{\frac13\right\}}$, $H_{\left\{\frac{\sqrt{2}}{2}\right\}}$ and a subgroup that is isomorphic to the Thompson group $F_3$
(which consists of piecewise linear functions $[0,1]\to [0,1]$ where slopes are powers of $3$  and break points are 3-adic). In \cite{GolanSapir} we also showed how to implicitly construct many other maximal subgroups of $F$, but we do not know whether these subgroups are isomorphic to each other or to some of the 4 maximal subgroups listed above.
It is quite possible, that one can construct copies of $F_n$ as maximal subgroups of infinite index in $F$ for every $n>2$ . Thus up to isomorphism, there should be at least countably many maximal subgroups of $F$ of infinite index.

\begin{Problem} Is the set of isomorphism classes of maximal subgroups of $F$ countable?
\end{Problem}

Note that a similar problem is interesting for many other groups (say, $\mathrm{SL}_n(\Z)$ whose maximal subgroups have been extensively studied by Margulis and Soifer). Note also that for some finitely generated groups, say, direct products of two non-cyclic free groups, the set of isomorphism classes of maximal subgroups is uncountable \cite{MO1}.

\subsection{Distortion and closed subgroups}

Closed subgroups of $F$ have been defined in \cite{GolanSapir}. The first author \cite{Golan16} showed that one can alternatively define closed subgroups as follows.

\begin{Definition} Let $h\in F$. Then {\em components} of $h$ are all elements of $F$ that coincide with $h$ on a closed interval $[a,b]$ with finite dyadic $a,b$ and are identity outside $[a,b]$ (in that case $h$ necessarily fixes $a$ and $b$). A subgroup $H\le F$ is \emph{closed} if for any $h\in H$, the subgroup $H$ contains all components of $h$.
\end{Definition}

It is clear from this definition that for every (not necessarily finite) set $U\subseteq (0,1)$ the stabilizer $H_U$ of $U$ is closed. It is quite possible that Theorem \ref{t:dist} can be generalized to arbitrary closed subgroups of $F$.

\begin{Problem}\label{p:56} Is it true that every finitely generated closed subgroup of $F$ is undistorted?
\end{Problem}

Note that it is an open problem (see \cite{GolanSapir}) whether all maximal subgroups of infinite index in $F$ are closed. If the answer is ``yes'' and the answer to Problem \ref{p:56} is affirmative, then all finitely generated maximal subgroups of $F$ would be undistorted (because subgroups of finite index in any finitely generated group are obviously undistored).

Note also that by \cite{GolanSapir} the distortion function of every finitely generated closed subgroup of $F$ is recursive because the membership problem for every such subgroup is decidable.

\subsection{Subgroups of quasi-finite index} The examples of Jones' subgroup $\overrightarrow F$ and Savchuk's subgroups $H_U$ show that subgroups of $F$ are of quasi-finite index surprisingly often. Using the 2-core of subgroups of $F$ as defined in \cite{GolanSapir}, one can construct more examples of subgroups of $F$ of infinite index which have quasi-finite index in $F$.

\begin{Problem} Is there an algorithm to check if a finite set of elements of $F$ generates a maximal subgroup or a subgroup of quasi-finite index?
\end{Problem}

Note that the first author has found an algorithm checking whether a finite set of elements of $F$ generates the whole $F$ \cite{Golan16}.

%\subsection{Conjugacy in $\Homeo(\R)$}
%
%In view of Theorem \ref{t:pl} it is natural to ask when an isomorphism between two subgroups of $F$ is induced by conjugacy in  $\Homeo(\R)$ or in $\F$. Here is  a typical problem.
%
%\begin{Problem} Let us call two embeddings of $F$ into $F$ \emph{equivalent} if the two images of $F$ are conjugate in $\Homeo(\R)$. Are any two embeddings of $F$ into $F$ such that the images act transitively on the set of finite binary fractions from $[0,1]$ equivalent? One can ask a similar question replacing $\Homeo(\R)$ by $\F$.
%\end{Problem}

\end{document}